%% file: Paper.tex
\documentclass[12pt,twoside]{article}

	\usepackage{fancyhdr} 
	\usepackage[margin=1in]{geometry}  
	\usepackage{graphicx}              
	\usepackage{amsmath}               
	\usepackage{amsfonts}              
	\usepackage{amsthm}                
	\usepackage{enumitem}
	\usepackage{amssymb}
	\usepackage{mathtools}
	\usepackage{hyperref}
	\usepackage{listings}
	\usepackage{xcolor}
	\usepackage{booktabs}
	\usepackage{enumitem}
	\ExplSyntaxOn
	\cs_new_eq:NN \intfromalph \int_from_alph:n
	\ExplSyntaxOff

    \DeclareMathOperator{\Gal}{Gal}
	\newtheorem{thm}{Theorem}[section]
	\newtheorem{lem}[thm]{Lemma}
	\newtheorem{prop}[thm]{Proposition}
	\newtheorem{cor}[thm]{Corollary}
	
	\theoremstyle{remark}
	\newtheorem{rem}[thm]{Remark}
	
	\theoremstyle{definition}
	\newtheorem{Def}[thm]{Definition}
	\theoremstyle{plain}
	\newtheorem*{thmA}{Theorem A}
	\newtheorem*{thmB}{Theorem B}
	\newtheorem*{thmC}{Theorem C}
    \newtheorem*{thmD}{Theorem D}

\begin{document}

		\nocite{*}
		
	\title{On the Comparison of the D-new Modular Degree and the Shimura Degree of Modular Abelian Varieties}
	\author{Mohammad Masih Hamidi}
	\date{}
	
	\maketitle
	
	\begingroup
	\renewcommand{\thefootnote}{}
	\footnotetext{%
		\parbox[t]{0.8\textwidth}{%
			}}
	\addtocounter{footnote}{-1}
	\endgroup

	\input{Abstract}
	\input{Introduction}
	\input{Strategy}
	\input{Proof_of_A}
	\input{Multiplicity_one}
	\input{Proof_of_B}
	\input{Case_of_l=2_and_a_Criterion}

	\input{Higher_dim}
	\input{Applications}

    \input{Appendix}
	\bibliography{paper}	
	\bibliographystyle{plain}
	\author{Mohammad Masih Hamidi}~~ \href{mohammad.hamidi@mail.mcgill.ca}{mohammad.hamidi@mail.mcgill.ca}\\ 
	Department of Mathematics, McGill University\\
	Address: 805 Sherbrooke Street West, Montreal, QC, Canada

			\end{document}

%% file: Abstract.tex
\begin{abstract}
	Let $E$ be an elliptic curve over $\mathbb{Q}$ of square-free
	conductor $N=DM$, where $D$ is a product of an even number of
	primes, and let $J^D(M)$ denote the Jacobian of the Shimura
	curve of discriminant $D$ and level $M$. Under mild assumptions
	on the local Galois representations of $E$, we prove that the
	Shimura degree of $E$ and its $D$-new modular degree have equal
	valuations at odd primes. Under the same assumptions, we
	establish the equality conjectured by Deines among the
	$D$-new modular degree, the $D$-new congruence number, and the
	Shimura degree of $E$ locally at odd primes. After introducing
	the Shimura congruence number into this framework, we show
	that all four quantities have equal valuations at such primes.
	Moreover, the two congruence numbers are always equal.
	
	By refining an argument of Helm, we establish new multiplicity
	bounds for the entire Shimura Jacobian, including in residue
	characteristic $3$. We use these bounds to obtain new cases of
	multiplicity one and prove our main theorems, following a
	strategy introduced by Agashe, Ribet, and Stein. Building on
	these results, we obtain new examples of the failure of
	multiplicity one. Finally, we give a negative answer to a
	question of Papikian and Rabinoff asking whether the functorial
	map between the component groups of the N\'eron models of
	$J^D(M)$ and $E$ is surjective when $D>1$, and we relate this
	phenomenon to the failure of multiplicity one.
\end{abstract}

%% file: Introduction.tex
\section{Introduction}
In their paper \cite{Agashe2012}, Agashe, Ribet, and Stein answered a question of Frey and M\"uller asking whether the modular degree of a rational elliptic curve $E$ coincides with its congruence number. Let $N$ be the conductor of $E$, and let $f\in S_2(N)$ be the normalized weight-two newform attached to $E$ by the modularity theorem. Here $S_2(N)$ denotes the space of weight-two cusp forms for $\Gamma_0(N)$. Roughly speaking, the congruence number measures congruences between $f$ and other integral cusp forms of the same weight and level, whereas the modular degree is the degree of the associated optimal modular parametrization. They showed that the modular degree always divides the congruence number, and that for elliptic curves with square-free conductors they are equal. If the conductor is not square-free, the two are not equal in general.

Let $N=DM$ be a square-free natural number, with $D$ being a product of an even number of primes. Let $J^D(M)$ be the Jacobian of the Shimura curve of level $M$ attached to an indefinite quaternion algebra of discriminant $D$. We write $J^1(N)=J_0(N)$ for the classical modular Jacobian and $J^1_D(N)$ for its $D$-new subvariety. The Jacquet--Langlands correspondence relates the relevant Hecke eigensystems on $J^D(M)$ and $J^1_D(N)$ (see Section~2). Let $\mathbb{T}^D(M)$ denote the Hecke algebra acting on $J^D(M)$, and let $A^D(M)$ denote the unique rational elliptic curve isogenous to $E$ appearing as an optimal subvariety of $J^D(M)$.

Write $r^1(N)$ and $\delta^1(N)$ for the usual congruence number
and modular degree of $E$ relative to $J_0(N)$, respectively.
It is natural to consider the analogous congruence problem
in which only cusp forms in the $D$-new subspace of $S_2(N)$
are allowed, leading to the $D$-new congruence number
$r^1_D(N)$. Although Shimura curves with $D>1$ have no cusps
and thus lack the usual $q$-expansion description of
congruences, the formulation of congruence numbers in terms
of ideals of Hecke algebras carries over naturally to this
setting. This leads us to introduce the Shimura congruence
number $r^D(M)$ attached to $A^D(M)$.
We also consider the corresponding Shimura degree
$\delta^D(M)$ and $D$-new modular degree $\delta^1_D(N)$.
We refer to Definition~3.1 for the precise definitions of
these invariants and denote $E$ by $A^1_D(N)$ when viewed
in $J^1_D(N)$. In view of the Jacquet--Langlands correspondence, it is
natural to compare these invariants on the two sides.
Since the correspondence identifies the relevant Hecke
algebras and matches the corresponding Hecke eigensystems,
one expects the two congruence numbers to coincide.
Our first result confirms this expectation.

\begin{thmA}
	$r^D(M)=r^1_D(N)$.
\end{thmA}

The comparison of these congruence numbers with the
corresponding degrees is more subtle. While the congruence
numbers measure the obstruction to separating the eigensystem
attached to $E$ from the other eigensystems integrally in
the Hecke algebra, the degrees measure the obstruction to
splitting off the corresponding elliptic curve integrally
from the ambient abelian variety. Both invariants can be
described in terms of the denominator of the same rational
Hecke idempotent, but in different rings. Moreover, the
Jacquet--Langlands correspondence relates the underlying
abelian varieties only up to isogeny and therefore does not
directly account for their integral structures.
Nevertheless, in analogy with the equality of the classical
modular degree and congruence number in the square-free case,
one might expect a similar equality in the present setting.
Indeed, Deines~(\cite{deines2014shimura},
Conjectures~5.1.7 and 5.1.8) conjectured that
\[
r^1_D(N)=\delta^1_D(N)=\delta^D(M).
\]
We use the newly defined invariant $r^D(M)$, Theorem~A,
and multiplicity-one results for the relevant Hecke modules
to establish this equality under suitable local hypotheses.
The role of multiplicity one is to control the saturation
of the Hecke algebra in the relevant endomorphism rings,
thereby relating the algebraic and geometric invariants.
This strategy was first used in the classical setting of
$J_0(N)$ by Agashe, Ribet, and Stein~\cite{Agashe2012}.

\begin{thmB}
	Let $l$ be an odd prime, let $\mathfrak{m}$ be the
	maximal ideal of $\mathbb{T}^D(M)$ attached to the
	residual Hecke eigensystem of $A^D(M)$ modulo $l$,
	and let $\bar{\rho}_{\mathfrak{m}}$ be the representation
	of $\Gal(\overline{\mathbb{Q}}/\mathbb{Q})$
	attached to $A^D(M)[l]$. Then
	\[
	\mathrm{ord}_l(r^1_D(N))
	=\mathrm{ord}_l(\delta^1_D(N))
	\geq \mathrm{ord}_l(\delta^D(M)).
	\]
	Moreover, if $\bar{\rho}_{\mathfrak{m}}$ is absolutely
	irreducible and $\mathfrak{m}$ is controllable
	(Definition~5.2) at all but possibly one prime dividing
	$D$, then the orders at $l$ of $\delta^D(M)$,
	$\delta^1_D(N)$, $r^D(M)$, and $r^1_D(N)$ are equal.
\end{thmB}

The reader may consult the end of this section for an
outline of the proof of Theorem~B.

Component groups play a central role in the comparison of Shimura degrees and congruence numbers through the work of Ribet and Takahashi \cite{ribet1997parametrizations}. For an abelian variety $B$ over $\mathbb{Q}$, let $\Phi_p(B)$ denote the group of connected components of the special fiber of its Néron model at $p$, and let $\mathcal X_p(B)$ denote the character group of the toric part of that special fiber. For a prime $p\mid D$, the dual of the inclusion $A^D(M)\hookrightarrow J^D(M)$ induces a map $\phi: \Phi_p(J^D(M))\to\Phi_p(A^D(M))$. The following result shows that this map need not be surjective, answering a question raised by Papikian and Rabinoff~\cite{papikian2016optimal}.

\begin{thmC}
	The map $\phi: \Phi_p(J^D(M)) \rightarrow
	\Phi_p(A^D(M))$ is not necessarily surjective.
\end{thmC}

To study the obstruction to surjectivity, we introduce the
notions of the \emph{lattice congruence ideal} and the
\emph{lattice intersection ideal} associated with $A^D(M)$.
These are analogues of the congruence and intersection ideals
introduced in Definition~3.1, but are defined in terms of
character groups rather than the underlying abelian varieties.
Using these new invariants, we establish a sufficient condition
for surjectivity in terms of character-group multiplicity one
(Section~9.2.2). Here, a subscript $\mathfrak{m}$ denotes
localization at a maximal ideal $\mathfrak{m}$ of the Hecke
algebra.

\begin{thmD}
	Let $A^D(M)$ be an elliptic curve attached to a newform $f$,
	viewed as a subvariety of $J^D(M)$, and let $p$ be a prime
	dividing $D$. Let $l$ be a prime number and let
	$\mathfrak{m}$ be the maximal ideal of $\mathbb{T}^D(M)$
	attached to $A^D(M)[l]$. If
	\[
	\mathcal{X}_p(J^D(M))_{\mathfrak{m}}
	\cong \mathbb{T}^D(M)_{\mathfrak{m}},
	\]
	then the map
	\[
	\phi: \Phi_p(J^D(M)) \rightarrow \Phi_p(A^D(M))
	\]
	is surjective after localization at $\mathfrak{m}$.
\end{thmD}

\begin{rem}
	After this work was completed, Papikian pointed out a
	possible connection between Theorem~D and Corollary~3.6
	of Emerton~\cite{emerton2003optimal}. It appears that
	Theorem~D may also follow from Emerton's result after
	identifying the character group of the Raynaud torus
	with that of the toric part of the N\'eron model and
	verifying the relevant hypotheses. We nevertheless give
	an independent proof that relates the above lattice ideals
	to the denominators of the idempotent attached to $A^D(M)$
	in various endomorphism rings. The connection between
	these denominators and component-group surjectivity is
	developed further in a more general setting in our
	subsequent work~\cite{hamidi2026hecke}.
\end{rem}
Theorem A is proved in Section 4, and we use it to conclude that $\delta^D(M) \mid \delta^1(N)$. In particular, the product of ratios appearing in the Ribet-Takahashi formula (\cite{ribet1997parametrizations}, Theorem A), which relates different Shimura degrees, is an integer (Corollary \ref{cor4.2}). The key is to use the Jacquet-Langlands correspondence. 

To prove Theorem~B, we first establish Proposition~5.5 in
Section~5.2, which, through a refinement of Helm's argument
in \cite{helm2007maps}, extends his multiplicity bound from
the $M$-new subvariety to the entire Jacobian $J^D(M)$ and
also allows for residue characteristic $3$.
This proposition bounds the dimension of \(J^D(M)[\mathfrak m]\)
in terms of the number of certain ``bad'' primes dividing
\(D\) (see Definition~5.2).
The related notion of controllability appearing in Theorem~B
is the local condition used to obtain the required
multiplicity-one bounds.
Moreover, Corollary~5.6 gives an explicit upper bound for
the dimension of \(J^D(M)[\mathfrak m]\) when the residue
characteristic of \(\mathfrak m\) is odd and does not divide
\(DM\).
For a divisor $Q\mid M$, let $J^D_Q(M)$ denote the subvariety
new at all primes dividing $Q$. Both results also extend
to these $Q$-new subvarieties (see Section~5.2).

Using Proposition 5.5, in Section 6, we first conclude that the Shimura congruence number and Shimura degree are equal locally at odd primes $l$, under the conditions of Theorem B. On the other hand, we will prove the equality of the $D$-new modular degree and $D$-new congruence number with no local condition at odd primes. Thus, by Theorem A, all four numbers are equal, i.e. Theorem B is established. 

In Section 7, we focus on the case $l=2$, which is more subtle because the multiplicity-one results used for odd $l$ are not available in the same generality. Proposition 7.1 proves that the order at $2$ of the Shimura degree is less than or equal to the order at $2$ of the $D$-new modular degree for a large family of elliptic curves. Using the results of previous sections, we formulate a criterion to check the failure of multiplicity one. 

Section 8 describes how results of the previous sections can be generalized to higher-dimensional modular abelian varieties attached to newforms of level $N$ and weight 2. In particular, we will generalize a result of \cite{takahashi2001degrees} about divisibility of different Shimura degrees. Then, we formulate an analogue of Proposition 7.2, with an additional statement by proving a generalization of the main result of \cite{agashe2008modular}. 

In the last section, we first use the aforementioned criteria
to obtain new examples of the failure of multiplicity one
for $J^D(M)$ and $J^1_D(N)$, the smallest conductor among
them being $105$ (Section~9.1). For instance, there exists
a maximal ideal $\mathfrak{m}$ of $\mathbb{T}^{35}(3)$
of residue characteristic $2$ such that
$
\dim_{\mathbb{T}^{35}(3)/\mathfrak{m}}
J^{35}(3)[\mathfrak{m}] > 2.
$
We then combine the results of the previous sections with
a formula of Papikian and Rabinoff~\cite{papikian2016optimal},
expressing the order of the cokernel of the map on component
groups as the ratio of certain denominators of the idempotent
attached to the optimal quotient, to prove Theorems~C and~D
(Section~9.2). Our argument is inspired by the strategy used
in~\cite{Agashe2012}, where multiplicity one is used
to control the saturation of the Hecke algebra in an
appropriate endomorphism ring.

\section{Notation and Background}

In this paper, we fix a square-free natural number $N=DM$, where $D$ is a product of an even number of primes ($D=1$ corresponds to the classical modular curve $X_0(M)$ of level $M$ and its Jacobian $J_0(M)$). We denote by $X^D(M)$ and $J^D(M)$ the Shimura curve of level $M$ over $\mathbb{Q}$ attached to an indefinite quaternion algebra of discriminant $D$ over $\mathbb{Q}$ and its Jacobian, respectively. Moreover, for any $Q$ dividing $M$, $J^D_Q(M)$ denotes the $Q$-new subvariety of $J^D(M)$. One then has a Hecke algebra $\mathbb{T}^D(M) \subseteq \text{End}(J^D(M))$ generated by the Hecke operators $T_q$ for prime numbers $q$. The subvariety $J^D_{Q}(M)$ is known to be $\mathbb{T}^D(M)$-invariant. Therefore, one can define the $Q$-new quotient of $\mathbb{T}^D(M)$, denoted by $\mathbb{T}^D_{Q}(M)$, as the image of $\mathbb{T}^D(M)$ in $\mathrm{End}(J^D_{Q}(M))$. The Jacquet-Langlands correspondence \cite{jacquet2006automorphic}, together with the work of Ribet \cite{ribet1990modular} and Faltings \cite{faltings1986finiteness}, imply that the natural map $\mathbb{T}^{D}(M) \rightarrow \mathbb{T}^1_{D}(N)$, sending the Hecke operator $T_q$ to the corresponding operator $T_q$ on the other side, is an isomorphism of rings, and that there is a Hecke-equivariant isogeny 
\begin{equation}
	\Phi: J^D(M) \rightarrow J^{1}_{D}(N).
\end{equation}
\begin{rem}
	The isogeny $\Phi$ is not necessarily unique.
\end{rem}
Given an elliptic curve $E$ over $\mathbb{Q}$ of conductor $N$, the above isogeny and the modularity theorem give a surjection of abelian varieties
\begin{equation}
	\pi: J^D(M) \rightarrow E.
\end{equation}
In the isogeny class of $E$, there is a unique elliptic curve $A^D(M)$ for which the corresponding surjection (which we again denote by $\pi$) has connected kernel. Using the autoduality of the abelian varieties $J^D(M)$ and $A^D(M)$, it is known that the map
\[
\hat{\pi} \circ \pi \in \mathrm{End}(A^D(M))
\]
is multiplication by an integer $\delta^D(M)$, called the Shimura degree of $E$.
\begin{rem}
	For $D=1$, $\delta^1(N)$ equals the degree of the classical modular parametrization $X_0(N) \rightarrow A^1(N)=E$. 
\end{rem}
Now, let $l$ be a prime dividing $N$. The following theorem describes the N\'eron model of $J^D(M)$ over $\mathbb{Z}_l$.
\begin{thm}
	\begin{enumerate}
		\item
		If $l$ divides $M$ (due to Deligne--Rapoport \cite{deligne1973schemas} and Buzzard \cite{buzzard1997integral}), then the fiber at $l$ of the N\'eron model $\mathcal{J}^D(M)$ of $J^D(M)$ sits in the following exact sequence:
		\begin{equation}
			0\rightarrow \mathcal{J}^D(M)^0_{/\mathbb{F}_l} \rightarrow \mathcal{J}^D(M)_{/\mathbb{F}_l} \rightarrow \Phi_l(J^D(M)) \rightarrow 0,
		\end{equation}
		where $\mathcal{J}^D(M)^0_{/\mathbb{F}_l}$ and $\Phi_l(J^D(M))$ are the connected component of the identity and the group of connected components of the special fiber, respectively.
		$\mathcal{J}^D(M)^0_{/\mathbb{F}_l}$ is an extension of two copies of the abelian variety $\mathcal{J}^D(M/l)_{/\mathbb{F}_l}$ by a torus $T$, as in the following exact sequence:
		\begin{equation}
			0 \rightarrow T \rightarrow \mathcal{J}^D(M)^0_{/\mathbb{F}_l}
			\rightarrow \mathcal{J}^D(M/l)_{/\mathbb{F}_l}
			\times \mathcal{J}^D(M/l)_{/\mathbb{F}_l}
			\rightarrow 0.
		\end{equation}
		\item
		If $l$ divides $D$ (due to Cerednik \cite{vcerednik1976uniformization} and Drinfeld \cite{drinfel1976coverings}), then we still have the exact sequence $(3)$, but $\mathcal{J}^D(M)^0_{/\mathbb{F}_l}$ is equal to $T$. In other words, the N\'eron model has purely toric reduction in this case.
	\end{enumerate}
\end{thm}

In either case of the above theorem, we let $\mathcal{X}_l(J^D(M))=\mathrm{Hom}(T,G_m)$ be the character group of the corresponding torus. Also, with $E$ and $Q$ as before, $\mathcal{X}_l(J^D_Q(M))$ and $\mathcal{X}_l(A^D(M))$ denote the character groups of $J^D_Q(M)$ and $A^D(M)$.

%% file: Strategy.tex
	\section{Congruence and Modular Numbers}
	
	In this section, we generalize the set-up introduced in \cite{Agashe2012} to the Jacobians of Shimura curves.
	We first generalize the notions of congruence exponent and modular exponent for elliptic curves to abelian varieties arising from Shimura Jacobians and their new subvarieties.
	\begin{Def}
		
		With notation as in Section 2, we can write $J^D_Q(M)$ uniquely as a sum of an elliptic curve $A^D_Q(M)$, isogenous to $E$, and the abelian variety $B^D_Q(M)=IJ^D_Q(M)$, where $I=\ker(\pi_{A^D_Q(M)}:\mathbb{T}^D_Q(M)\rightarrow \mathrm{End}(A^D_Q(M)))$.
		Similarly, one has a map
		$\pi_{B^D_Q(M)}:\mathbb{T}^D_Q(M)\rightarrow \mathrm{End}(B^D_Q(M))$,
		and let $\mathbb{T}^D_Q(M)_{A^D_Q(M)}$ and
		$\mathbb{T}^D_Q(M)_{B^D_Q(M)}$ be
		$\pi_{A^D_Q(M)}(\mathbb{T}^D_Q(M))=\mathbb{Z}$
		and
		$\pi_{B^D_Q(M)}(\mathbb{T}^D_Q(M))$,
		respectively.

		\begin{enumerate}
	    	\renewcommand{\labelenumi}{(\roman{enumi})}
			\item
			The \emph{congruence ideal} attached to $A^D_Q(M)$ is
			\[
			R^D_Q(M):=\pi_{A^D_Q(M)}\bigl(\ker(\pi_{B^D_Q(M)})\bigr)
			\subseteq \mathbb{Z}.
			\]
			The \emph{congruence exponent}, respectively the \emph{congruence number}, attached to $A^D_Q(M)$ are
			\[
			r^D_Q(M)=\exp\bigl(\mathbb{Z}/R^D_Q(M)\bigr),
			\qquad
			\tilde r^D_Q(M)=\#\bigl(\mathbb{Z}/R^D_Q(M)\bigr);
			\]
			
		\item	The \emph{intersection ideal} attached to $A^D_Q(M)$ is
			\[
			S^D_Q(M):=
			\operatorname{Ann}_{\mathbb{Z}}
			\bigl(A^D_Q(M)\cap B^D_Q(M)\bigr).
			\]
			The \emph{modular exponent}, respectively the \emph{modular number}, attached to $A^D_Q(M)$ are
			\[
			\delta^D_Q(M)=
			\exp\bigl(A^D_Q(M)\cap B^D_Q(M)\bigr),
			\qquad
			\tilde{\delta}^D_Q(M)=
			\#\bigl(A^D_Q(M)\cap B^D_Q(M)\bigr).
			\]
			
		\end{enumerate}
	\end{Def}
	
     Here, $R^1_1(N)$ and $S^1_1(N)$ are just $R$ and $S$ in Definition~5.1 of \cite{Agashe2012}. From now on, we drop the subscript $Q$ when $Q=1$.

	\begin{rem}
		Since $J^D(M)$ is self-dual, it follows that $\# (A^D(M)\cap IJ^D(M))$ is a perfect square, and that $\delta^D_Q(M)=\sqrt{\tilde{\delta}^D_Q(M)}$. However, in general, $\# (A^D_{Q}(M)\cap IJ^D_{Q}(M))$ is not necessarily a perfect square. For example, for the isogeny class 238a in the Sage database \cite{stein2007sage}, using the code in the appendix of \cite{deines2014shimura}, we have computed that
		$\#(E\cap IJ^1_{119}(238))=8$. This indicates that $J^1_{119}(238)$ is not self-dual; otherwise, let $\Psi $ be a principal polarization of it. Using strong multiplicity one for newforms, it is not hard to see that $E\cap IJ^1_{119}$ is the kernel of the endomorphism $E \hookrightarrow J^1_{119}(238) \xrightarrow{\Psi} {J^1_{119}(238)}^{\vee} \twoheadrightarrow E^{\vee} \xrightarrow{\cong} E$, which is multiplication by some positive integer on $E$ (since it is defined over $\mathbb{Q}$). Thus, $\# (E\cap IJ^1_{119})$ is a perfect square, which is a contradiction. Instead, one can define the $D$-new modular degree as the exponent of $E\cap IJ^1_{D}$, which agrees with Definition 3.1, i.e. the size of $\mathbb{Z}/S^1_{D}(N)$.
	\end{rem}
	
We now define the multiplicity one properties and describe how they are related to the equality of the modular exponent and the congruence exponent of the elliptic curve. Let $l$ be a prime number dividing $DM$, and let $\mathcal{J}^D_{\mathbb{Q}}(M)$ denote the N\'eron model of $J^D_{\mathbb{Q}}(M)$ over $\mathbb{Z}_{l}$. We write $\mathrm{Cot}_{\mathbb{F}_l}(\mathcal{J}^D_{\mathbb{Q}}(M))$ for the cotangent space at the identity section of the special fiber $\mathcal{J}^D_{\mathbb{Q}}(M)_{\mathbb{F}_l}$, and $\mathcal{X}_l(J^D_{\mathbb{Q}}(M))$ for the character group of the toric part of the identity component of this special fiber.
	
	\begin{Def}\label{1.8.}
		Let $\mathfrak{m}$ be a maximal ideal of $\mathbb{T}^D_Q(M)$ of residue characteristic $l$. We say $\mathfrak{m}$ satisfies
		\begin{enumerate}
			\renewcommand{\labelenumi}{(\roman{enumi})}
			\item 
			\emph{multiplicity one for differentials}, if 
			\begin{equation*}
				\text{dim}_{\mathbb{T}^D_Q(M)/\mathfrak{m}}\text{Cot}_{\mathbb{F}_l}(\mathcal{J}^D_Q(M))[\mathfrak{m}]\leq 1;
			\end{equation*}
			\item 
			\emph{multiplicity one for the Jacobian}, if
			\begin{equation*}
				\text{dim}_{\mathbb{T}^D_Q(M)/\mathfrak{m}}(J^D_Q(M)[\mathfrak{m}])=2;
			\end{equation*}
			\item 
	    	\emph{multiplicity one for characters}, if
	     	\[
	         \dim \mathcal{X}_l(J^{D}_{Q}(M))/
		      \mathfrak{m}\mathcal{X}_l(J^{D}_{Q}(M))
	           	=1,
		    \]
	       	and if $l$ divides $D$ or $Q$.
		\end{enumerate}
	\end{Def}
	
	\begin{prop}
		Let $l$ be a prime number. If the maximal ideal
		$\mathfrak{m}\subseteq \mathbb{T}^D_{Q}(M)$
		of characteristic $l$, attached to the elliptic curve $A^D_{Q}(M)$, satisfies either of the multiplicity one types above, then
		$\operatorname{ord}_{l}(r^D_Q(M))
		=
		\operatorname{ord}_{l}(\delta^D_Q(M))$.
	\end{prop}
	
	\begin{proof}
		Let $\mathbb{T}^D_Q(M)'$ be the saturation of $\mathbb{T}^D_Q(M)$ in End($J^D_Q(M)$), i.e. $\mathbb{T}^D_Q(M)'=\mathbb{T}^D_Q(M)\otimes \mathbb{Q}~\cap$ End($J^D_Q(M)$). Now, using Pontryagin duality, as well as Nakayama's lemma, cases (i), (ii), and (iii) in Definition 3.3 imply that
		${\text{Tan}_{\mathbb{Z}_l}(\mathcal{J}^D_Q(M))}_\mathfrak{m} \cong \mathbb{T}^D_Q(M)_{\mathfrak{m}}$,  $Ta_{\mathfrak{m}}(J^D_Q(M)) \cong \mathbb{T}^D_Q(M)_{\mathfrak{m}}^2 $, and $\mathcal{X}_l(J^{D}_{Q}(M))_{\mathfrak{m}} \cong \mathbb{T}^D_Q(M)_{\mathfrak{m}}$, respectively. Here, $\mathrm{Tan}_{\mathbb{Z}_l}(\mathcal{J}^D_Q(M))$ and $Ta_{\mathfrak{m}}(J^D_Q(M))$ denote the tangent space at the identity section of $\mathcal{J}^D_Q(M)$ and the $\mathfrak{m}$-adic Tate module of $J^D_Q(M)$, respectively. By the faithfulness of the action of $\mathrm{End}(J^D_Q(M))_{\mathfrak{m}}$ on $\text{Tan}_{\mathbb{Z}_l}(\mathcal{J}^D_Q(M))_{\mathfrak{m}}$, $Ta_{\mathfrak{m}}(J^D_Q(M))$, and $\mathcal{X}_l(J^{D}_{Q}(M))_{\mathfrak{m}}$ (faithfulness of the action on the character group follows from the fact that $J^D_Q(M)$ has purely toric reduction at primes dividing $D$ or $Q$), in either case, it follows that $\mathbb{T}^D_{Q}(M)'_{\mathfrak{m}}=\mathbb{T}^D_{Q}(M)_{\mathfrak{m}}$.
		
		The proof can be finished as follows: first note that, by Definition 3.1, $R^D_Q(M) \subseteq S^D_Q(M)$, and this induces a surjection $\mathbb{Z}/R^D_Q(M) \rightarrow \mathbb{Z}/S^D_Q(M)$. Hence, $\delta^D_Q(M) \mid r^D_Q(M)$. Moreover, it is not hard to see that
		$S^D_Q(M)=\mathbb{T}^D_Q(M)' \cap \pi_{A^D}(\mathbb{T}^D_Q(M))$
		and
		$R^D_Q(M)=\mathbb{T}^D_Q(M)\cap S^D_Q(M)$,
		where the intersection is taken in
		$\mathrm{End}(J^D_Q(M))\otimes \mathbb{Q}$. Therefore, we get an injection
		\[
		S^D_Q(M)/R^D_Q(M)=S^D_Q(M)/{(S^D_Q(M)\cap \mathbb{T}^D_Q(M))}\cong (S^D_Q(M)+\mathbb{T}^D_Q(M))/{\mathbb{T}^D_Q(M)} \xhookrightarrow{} \mathbb{T}^D_Q(M)'/\mathbb{T}^D_Q(M).
		\]
		Since $\mathbb{T}^D_Q(M)'_{\mathfrak{m}}=\mathbb{T}^D_Q(M)_{\mathfrak{m}}$, this implies that $R^D_Q(M)_{\mathfrak{m}}=S^D_Q(M)_\mathfrak{m}$. Hence, $\text{ord}_l(r^D_Q(M))=\text{ord}_l(\delta^D_Q(M))$, for $l$ the characteristic of $\mathbb{T}^D_Q(M)/\mathfrak{m}$.
	\end{proof}

%% file: Proof_of_A.tex
\section{Proof of Theorem A}
\begin{proof}[Proof of Theorem A]
	Recall that, by the results of Jacquet--Langlands, Faltings, and Ribet, there is a $\mathbb{T}^D(M)$-equivariant isogeny
	\[
	\phi : A^D(M)+B^D(M)= J^D(M) \rightarrow J^1_{D}(N)=\phi(A^D(M))+\phi(B^D(M)).
	\]
	Thus, the elliptic curve $\phi(A^D(M))$ is the optimal quotient of the dual of $J^1_D(N)$. Note that $\phi(B^D(M))=\tilde{I}J^1_D(N)$, where $\tilde{I}=\text{ker}(\pi_{\phi(A^D(M))})$. By definition, if $t \in R^D(M)=\pi_{A^D(M)}(\mathrm{ker}(\pi_{B^D(M)}))$, there is a $\tilde{t} \in \mathbb{T}^D(M)$ such that $\tilde{t}(a)=t(a)$ and $\tilde{t}(b)=0$ for any $a\in A^D(M)$ and any $b \in B^D(M)$. Now, consider this $\tilde{t}$ in $\mathrm{End}(J^1_D(N))$. For any $\phi(b) \in \phi(B^D(M))$, $\tilde{t}(\phi(b))=\phi(\tilde{t}(b))=\phi(0)=0$. Therefore, the identity morphism on $\mathbb{Z}$ maps $R^D(M)$ into $R^1_{D}(N)$. Hence, we get a natural ring surjection $\mathbb{Z}/R^D(M) \rightarrow \mathbb{Z}/R^1_{D}(N)$. In particular, $r^1_{D}(N) \mid r^D(M)$. Considering a Hecke-equivariant isogeny similar to $\phi$, but in the opposite direction, whose existence follows similarly, we get the divisibility in the other direction, and this finishes the proof of the theorem. 
\end{proof}

\begin{cor}\label{cor4.2}
	
$\delta^{D}(M) \mid \delta^1(N)$.
	
\end{cor}

\begin{proof}
	
	From the definition, it is clear that $\delta^D(M) \mid r^D(M)=r^1_{D}(N)$ (by Theorem A). Moreover, it can be easily seen that the quotient $\mathbb{T}^1(N) \rightarrow \mathbb{T}^1_D(N)$ maps $R^1(N)$ into $R^1_{D}(N)$. Hence, we have $r^1_{D}(N)\mid r^1(N)$. On the other hand, it is the main result of \cite{Agashe2012} that $r^1(N)=\delta^1(N)$, and thus $\delta^D(M) \mid \delta^1(N)$.
\end{proof}

\begin{rem}
	Corollary~4.1 was proved in Theorem~3.2 of \cite{takahashi2001degrees} using an explicit global formula. However, that proof relies on a result of Bertolini and Darmon which, as noted on page 1364 of \cite{papikian2016optimal}, implicitly assumes the surjectivity of a certain map on the component groups of N\'eron models. In Theorem C, we show that this surjectivity fails in general.
\end{rem}

%% file: Multiplicity_one.tex
\section{Multiplicity One}
The multiplicity one phenomena introduced in the introduction are very helpful in many places. The strategy introduced by Agashe, Ribet, and Stein in \cite{Agashe2012} concludes the equality of the modular degree and the congruence number attached to a rational semistable elliptic curve via multiplicity one for both differentials and the Jacobian. It is natural to look at the Shimura case and the $D$-new subvariety of the Jacobian of the modular curve to see if we have multiplicity one there and then obtain similar results. The purpose of this section is to discuss the known cases of multiplicity one, as well as proving new cases.

\subsection{Known Cases}
Here, we state known multiplicity one results, which will later be specialized to the case of elliptic curves. Let $N=DM$ as before, and let $\mathfrak{m}$ be a maximal ideal of $\mathbb{T}^D(M)$ such that the associated representation $\bar{\rho}_{\mathfrak{m}}$ is irreducible. The irreducibility condition is often needed in the proofs, and it is usually the case for the maximal ideals we are interested in. The following theorem summarizes the important cases of Jacobian multiplicity one in the case $D=1$:

\begin{thm}
	Let $N$ be a square-free natural number. Suppose $\mathfrak{m}$ is a maximal ideal of $\mathbb{T}^1(N)$, of characteristic $l$, and the associated representation $\bar{\rho}_{\mathfrak{m}}$ is absolutely irreducible. In the following cases, dim~$(J^1(N)[\mathfrak{m}])=2$.
	
	\begin{enumerate}[label=(\alph*),start=\intfromalph{a}]
		\item 
		$l$ is odd and does not divide $N$ (\cite{ribet1990modular});
		\item 
		$l$ is odd, divides $N$, and $a_l \neq 0$, where $a_l$ is the reduction mod $l$ of the $l$-th Fourier coefficient of the newform $f$ attached to $\mathfrak{m}$, and $\mathfrak{m}$ is $G_l$-distinguished (see \cite{tilouine1997hecke} for definitions);
		\item
		$l$ divides $N$ and $\mathfrak{m}$ is not $l$-old (\cite{mazur1991two});
		\item 
		$l=2$, $N$ is odd, and $\bar{\rho}_{\mathfrak{m}}$ restricted to a decomposition group at $2$ is not contained within the scalar matrices (\cite{article}).
	\end{enumerate}
\end{thm}

There is a connection between differential and Jacobian multiplicity one. For example, in case (a) of Theorem 5.1, the $q$-expansion principle is used to deduce that $$\mathrm{dim}(H^0(X_0(N)),\Omega^1)[\mathfrak{
	m}]\leq 1,$$ 
which itself is used to conclude that $$\mathrm{dim}(J^1[\mathfrak{
	m}])=2.$$ For the maximal ideals coming from elliptic curves, as crucially used in \cite{Agashe2012}, multiplicity one for differentials is known since $N$ is square-free.

Now, let us mention the result we will use for Jacobian multiplicity one in the Shimura Jacobian case. There are several approaches, but we only focus on the one obtained by Helm in \cite{helm2007maps}, mainly because we are working over $\mathbb{Q}$, and it suffices for our purposes. However, Helm's result gives a multiplicity bound for a certain new subvariety of the Shimura Jacobian, not the entire Jacobian. In the next section, we modify his argument and obtain a corresponding multiplicity bound for the entire Jacobian under the same assumptions that he imposed. We can state his main result~(Corollary 8.11 of \cite{helm2007maps}) after the following definition:

\begin{Def}
	Let $N=DM$ be a square-free natural number, and let $D$ be a product of an even number of primes. Let $\mathfrak{m}$ be a maximal ideal of $\mathbb{T}^D(M)$ such that the representation $\bar{\rho}_{\mathfrak{m}}$ is irreducible. We say $\mathfrak{m}$ is controllable at a prime $p$ dividing $N$ if any of the following conditions holds:
	\begin{enumerate}
		\item 
		$\bar{\rho}_{\mathfrak{m}}$ is ramified at $p$;
		\item 
		$\bar{\rho}_{\mathfrak{m}}$ is unramified at $p$, $p\neq l$, and $\bar{\rho}_{\mathfrak{m}}(\text{Frob}_p)$ is not a scalar;
		\item 
		$p=l$, and $l \neq 2$;
		\item 
		$p=l=2$, and the restriction of $\bar{\rho}_{\mathfrak{m}}$ to a decomposition group at $2$ is not contained in the scalar matrices.
	\end{enumerate}
\end{Def}  
The above condition ensures that $\mathcal{X}_p(J^{\text{min}}_D(M))$ is locally free (of rank 1) at $\mathfrak{m}$, where $J^{\text{min}}_D(M)$ is a pivotal abelian variety (see Lemma 4.5 of \cite{helm2007maps} for more details) isogenous to $J^D(M)$, which Helm used to prove multiplicity one at $\mathfrak{m}$ for the $M$-new subvariety of $J^D(M)$. 
\begin{thm}[Helm]
	Let $\mathfrak{m}$ be a maximal ideal of $\mathbb{T}^D_M(M)$ such that $\bar{\rho}_{\mathfrak{m}}$ is irreducible. Assume, moreover that theresidue characteristic of $\mathfrak{m}$ is at least $5$. Let $k$ be the number of primes dividing $D$ at which $\mathfrak{m}$ is not controllable. Then,
	\[
	\text{dim}~J^D_M(M)[\mathfrak{m}]\leq 2^k \text{dim}~J^1_N(N)[\mathfrak{m}].
	\] 
	In particular, let $\tilde{\mathfrak{m}}$ be the preimage of $\mathfrak{m}$ under the quotient map $\mathbb{T}^1(N) \rightarrow \mathbb{T}^1_N(N) \cong \mathbb{T}^D_M(M)$. If  $\mathfrak{m}$ is controllable at all primes dividing $D$ and $\mathrm{dim}~J^1(N)[\tilde{\mathfrak{m}}]=2$,
	then $\mathrm{dim}~J^D_M(M)[\mathfrak{m}]=2$.
\end{thm}

\begin{rem}
	In \cite{helm2007maps}, the above theorem is stated for an arbitrary congruence subgroup $\Gamma$ of level $M$, but we only need the case $\Gamma=\Gamma_0(M)$ in this paper.
\end{rem}

In the next section, we adapt his argument to obtain a multiplicity
bound for the entire Jacobian $J^D(M)$, valid also in residue
characteristic $3$.
\subsection{New Cases}

We start with the following variation of Theorem 5.3.

\begin{prop}
	
	\begin{enumerate}
		\item 	
		Let $l>2$ be a prime number and let $\mathfrak{m}$ be a maximal ideal of $\mathbb{T}^D(M)$ of characteristic $l$. Assume, moreover, that $\bar{\rho}_{\mathfrak{m}}$ is irreducible. Let $k$ be the number of primes dividing $D$ at which $\mathfrak{m}$ is not controllable. Then for any pair of primes $r$ and $s$ dividing $D$, let $\mathfrak{M}$ be the maximal ideal of $\mathbb{T}^1_{rs}(DM)$ which is the preimage of $\mathfrak{m}$ under the natural quotient map $\mathbb{T}^1_{rs}(DM) \rightarrow \mathbb{T}^1_{D}(DM)\cong\mathbb{T}^D(M)$. Then, $\text{dim}~J^D(M)[\mathfrak{m}]\leq 2^k\text{dim}~J^1_{rs}(DM)[\mathfrak{M}]$. In particular, if $\mathfrak{m}$ is controllable at all primes dividing $D$, $\text{dim}~J^D(M)[\mathfrak{m}]\leq \text{dim}~J^1_{rs}(DM)[\mathfrak{M}]\leq  \text{dim}~J^1(DM)[\mathfrak{M'}]$, where $\mathfrak{M'}$ is the preimage of $\mathfrak{m}$ under the natural quotient map $\mathbb{T}^1(DM) \rightarrow \mathbb{T}^1_{D}(DM)\cong\mathbb{T}^D(M)$;
		\item 
		More generally, let $Q$ be a divisor of $M$, and let $\mathfrak{n}$ be a maximal ideal of $\mathbb T^D_Q(M)$ and let $\tilde{\mathfrak{n}}$ be the preimage of $\mathfrak{n}$ under the natural quotient map $\mathbb{T}^{D/rs}_{Q}(rsM) \rightarrow \mathbb{T}^D_Q(M)$. Under the same assumptions and notation as in part~(1), after replacing $\mathfrak m$ by $\mathfrak n$, $\text{dim}~J^D_Q(M)[\mathfrak{n}]\leq 2^k\text{dim}~J^1_{rsQ}(DM)[\tilde{\mathfrak{n}}]$
	\end{enumerate}
	
\end{prop}

The main result of this section is the following:

\begin{cor}
	\begin{enumerate}
		\item
		Let $l>2$ be a prime not dividing $DM$, and let $\mathfrak m$ be a maximal ideal of $\mathbb T^D(M)$ of characteristic $l$. Assume that $\bar{\rho}_{\mathfrak m}$ is irreducible. Let $k$ denote the number of primes dividing $D$ at which $\mathfrak m$ is not controllable, in the sense of Definition~5.2. Then
		\[
		\dim J^D(M)[\mathfrak m]\leq 2^{k+1}.
		\]
		
		\item
		More generally, let $Q$ be a divisor of $M$, and let $\mathfrak n$ be a maximal ideal of $\mathbb T^D_Q(M)$. Under the same assumptions and notation as in part~(1), after replacing $\mathfrak m$ by $\mathfrak n$,
		\[
		\dim J^D_Q(M)[\mathfrak n]\leq 2^{k+1}.
		\]
	\end{enumerate}
\end{cor}
It is worth pointing out that a similar result is obtained in \cite{manning2021patching}, using the Taylor--Wiles--Kisin patching method, in the case $M=1$ but for an arbitrary totally real number field under some hypotheses~(see Theorem~1.1 of \emph{loc.\ cit.} for the precise statement).

Before proving the proposition, we recall the following crucial lemma appearing in \cite{helm2007maps}, which we use in the proof of the proposition.

\begin{lem}
	There exists an abelian variety $J^{\text{min}}_D(M)$, with a $\mathbb{T}^D(M)$-action, such that $J^D(M)$ and $J^{\text{min}}_D(M)$ are $\mathbb{T}^D(M)$-equivariantly isogenous (over
	$\mathbb{Q}$) with the following property:
	
	If $\mathfrak{m}$ is a maximal ideal of $\mathbb{T}^D(M)$ such that $\bar{\rho}_{\mathfrak{m}}$ is absolutely irreducible, then $J^{\text{min}}_D(M)[\mathfrak{m}]$ is two dimensional over $\mathbb{T}^D(M)/\mathfrak{m}$.
\end{lem}
\begin{proof}
	See Lemma 4.5 of \cite{helm2007maps}.
\end{proof}

With the abelian variety $J^{\text{min}}_D(M)$ as in the above lemma, let $J$ be any abelian variety over $\mathbb{Q}$ with a $\mathbb{T}^D(M)$-action, such that there exists a $\mathbb{T}^D(M)$-equivariant isogeny $\phi : J \rightarrow J^{\text{min}}_D(M) $ defined
over $\mathbb{Q}$. Let $[J]$ denote the $\mathbb{T}^D(M)$-module Hom($J$, $J^{\text{min}}_D(M))$. Also, let ${X}_p(J^{min}_{D}(M))$ be the character group of the torus appearing in the N\'eron model of $J^{\text{min}}_D(M)$ over $\mathbb{Z}_p$.

\begin{proof}[Proof of Proposition 5.5]
	Throughout this proof, tensor products are understood in the convention of \cite{helm2007maps}, i.e. the ordinary tensor product modulo its $\mathbb Z$-torsion. To make this convention explicit in formula~(6), if $P$ denotes the ordinary tensor product, we write $P_{\mathrm{tf}}=P/P_{\mathbb Z\text{-tors}}$.

For $\ell=3$, we explain how to address the
small-characteristic restrictions in Helm's argument
in our squarefree $\Gamma_0$-setting. The main issue
is the level-raising argument in
\cite[Section~7]{helm2007maps}, which relies on
results of Diamond--Taylor originally stated for
$\ell>3$. We use the Ihara lemma of Manning--Shotton
to overcome this restriction.
By
\cite[Theorem~6.5 and Remark~6.7]{manningShotton2021ihara},
the required Ihara lemma holds provided that
$
3\mid\#\bar\rho_{\mathfrak m}(G_{\mathbb Q}).
$
We first verify this condition.
Suppose, to the contrary, that
$
3\nmid\#\bar\rho_{\mathfrak m}(G_{\mathbb Q}).
$
For every prime $q\neq 3$, semistability implies
that the restriction of $\bar\rho_{\mathfrak m}$
to $I_q$ is unipotent. Thus
$\bar\rho_{\mathfrak m}(I_q)$ is a $3$-group,
and our assumption forces
$\bar\rho_{\mathfrak m}(I_q)=1$.
Consequently, $\bar\rho_{\mathfrak m}$ is
unramified outside $3$.
We next show that $\bar\rho_{\mathfrak m}$ is
finite flat at $3$. If $3\nmid N$, this follows
from good reduction. If $3\mid N$, the local
representation has the form
$
\bar\rho_{\mathfrak m}|_{G_{\mathbb Q_3}}
\sim
\begin{pmatrix}
	\psi\bar\chi_3 & *\\
	0 & \psi
\end{pmatrix},
$
where $\psi$ is unramified and $\bar\chi_3$
is the mod-$3$ cyclotomic character.
Since $\bar\rho_{\mathfrak m}|_{G_{\mathbb Q_3}}$ has order prime to $3$,
this extension
splits. The character $\psi$ admits a finite
\'etale model over $\mathbb Z_3$, while
$\psi\bar\chi_3$ is realized by an unramified
twist of $\mu_3$. Both are therefore finite
flat, and so is their direct sum.
We can now apply
\cite[Theorem~1.1]{ribet1990modular}.
If $3\mid N$, finite flatness allows us first
to remove $3$ from the level. The resulting
level is prime to $3$, and since
$\bar\rho_{\mathfrak m}$ is unramified at every
remaining prime dividing the level, successive
applications of the same theorem lower the
level to $1$. This contradicts
$
S_2(\Gamma_0(1))=0.
$
Hence
$
3\mid\#\bar\rho_{\mathfrak m}(G_{\mathbb Q}).
$
It follows that
\cite[Theorem~6.5]{manningShotton2021ihara}
provides the localized Ihara lemma needed for
Helm's level-raising argument. In particular,
it supplies the injectivity of the degeneracy
map used in
\cite[Proposition~7.2]{helm2007maps}.
For the auxiliary primes $q$ chosen as in
\cite[Lemma~7.1]{helm2007maps}, the same argument of Part 1 of the lemma gives
$
q\equiv-1\pmod 3,
~
\operatorname{tr}
\bar\rho_{\mathfrak m}(\operatorname{Frob}_q)=0.
$
Ihara's
lemma, together with an adaptation of the argument of \cite[Lemma~1 and~2]{ribet1990raising} to our Shimura-curve setting, then gives
the required $q$-new maximal ideal in Part~2 of Lemma 7.1 of \cite{helm2007maps}. Thus the
level-raising steps of Helm's argument extend
to residue characteristic $3$.
Finally, the other small-characteristic
obstruction, arising from component groups,
is absent for $\Gamma_0$-levels, as explained
in \cite[Remark~8.12]{helm2007maps}.

	\begin{enumerate}
		\item 
		Following the proof of Corollary 8.10 of \cite{helm2007maps}, it is enough to prove that $$\text{dim}~[J^D(M)]/\mathfrak{m}[J^D(M)]\leq 2^k \text{dim}~[J^1_{rs}(DM)]/\mathfrak{M}[J^1_{rs}(DM)].$$ as in Proposition 8.1 of \cite{helm2007maps}. Write $D'=D/rs$. By Proposition 8.1 of $loc.~cit.$, there is a $\mathbb{T}^D(M)$-module $L_D$ such that ${(L_D)}_{\mathfrak{m}} \cong \mathbb{T}^D(M)_{\mathfrak{m}}$, and that if $p$ and $q$ divide $D'$, 
		\begin{equation}
			[J^D(M)] \cong \mathcal{X}_p(J^{min}_{D}(M)) \otimes \mathcal{X}_q(J^{min}_D(M)) \otimes [J^{D/pq}_{pq}(pqM)] \otimes L_D.
		\end{equation}
		
		Here, instead of following the proof of Proposition~8.4 of \emph{loc.\ cit.} and tensoring both sides of~(5) with $\mathbb{T}^1_N(N)$, we use the identification of Lemma~8.2 of \emph{loc.\ cit.} to replace
		$
		[J^{D/pq}_{pq}(pqM)]
		$
		by
		$
		[J^{D/pq}(pqM)] \otimes \mathbb{T}^{D}(N).
		$
		This formulation makes the underlying $\mathbb{T}^{D}(M)$-module structure more transparent. Similar to Helm's construction, there is an abelian variety $J^{min}_{D/pq}(pqM)$ isogenous to $J^{D/pq}(pqM)$ with all properties proved in $loc.~ cit.$ In particular, the locally free (of rank 1) module 
		$L_D$ is replaced by the analogous one for discriminant $D/pq$, denoted by $L_{D/pq}$. Now, let $\tilde{\mathfrak{m}}$ be the preimage of $\mathfrak{m}$ under the natural quotient map $\mathbb{T}^{D/pq}(pqM) \cong \mathbb{T}^{1}_{D/pq}(DM) \rightarrow \mathbb{T}^{1}_D(DM)\cong \mathbb{T}^D(M)$. Following the proof of Proposition 8.1 of $loc.~ cit.$, we obtain a formula relating $[J^D(M)]$ to $[J^{D/pq}(pqM)]$ locally at $\tilde{\mathfrak{m}}$ as follows: (tensoring over $\mathbb{T}^{D/pq}(M)$)
		\begin{align*}
			[J^D(M)]_{\tilde{\mathfrak m}}
			&\cong
			\Bigl(
			\mathcal X_p(J_D^{\min}(M))
			\otimes
			\mathcal X_q(J_D^{\min}(M))
			\otimes
			[J^{D/pq}_{pq}(pqM)]
			\otimes
			L_D
			\Bigr)_{\tilde{\mathfrak m}}
			\\
			&\cong
			\Bigl(
			\mathcal X_p(J_D^{\min}(M))
			\otimes
			\mathcal X_q(J_D^{\min}(M))
			\otimes
			[J^{D/pq}(pqM)]
			\otimes
			\mathbb T^{D}(M)
			\otimes
			L_D
			\Bigr)_{\tilde{\mathfrak m}}
			\\
			&\cong
			\mathcal X_p(J_D^{\min}(M))_{\tilde{\mathfrak m}}
			\otimes
			\mathcal X_q(J_D^{\min}(M))_{\tilde{\mathfrak m}}
			\otimes
			[J^{D/pq}(pqM)]_{\tilde{\mathfrak m}}
			\otimes
			\mathbb T^{D}(M)_{\tilde{\mathfrak m}}
			\otimes
			(L_D)_{\tilde{\mathfrak m}}
			\\
			&\cong
			\mathcal X_p(J_D^{\min}(M))_{\mathfrak m}
			\otimes
			\mathcal X_q(J_D^{\min}(M))_{\mathfrak m}
			\otimes
			[J^{D/pq}(pqM)]_{\tilde{\mathfrak m}}.
		\end{align*}
		In the last step we have used that
		$\mathcal X_p(J_D^{\min}(M))$,
		$\mathcal X_q(J_D^{\min}(M))$,
		and $L_D$
		are naturally $\mathbb T^{D}(M)$-modules, so the action of
		$\mathbb T^{D/pq}(pqM)$ on these modules factors through the quotient
		$\mathbb T^{D}(M)$. We have also used the isomorphism $(L_D)_{\mathfrak m}\cong \mathbb T^{D}(M)_{\mathfrak m}$. Hence,
		\begin{equation}
			[J^D(M)]_{\mathfrak{m}} \cong \Bigl(\mathcal{X}_p(J^{min}_{D}(M))_{\mathfrak{m}} \otimes \mathcal{X}_q(J^{min}_D(M))_{\mathfrak{m}} \otimes [J^{D/pq}(pqM)]_{\tilde{\mathfrak{m}}}\Bigr)_{\mathrm{tf}}.
		\end{equation}
		
		Applying a standard Mazur principle argument, since $\text{dim}~J^{\text{min}}_D(M)[\mathfrak{m}]=2$, we have $$\text{dim}~\mathcal{X}_p(J^{min}_{D}(M))_{\mathfrak{m}}/\mathfrak{m}\mathcal{X}_p(J^{min}_{D}(M))_{\mathfrak{m}} \leq 2,$$ and this dimension is $1$ if $\mathfrak{m}$ is controllable at $p$ (see Lemma 6.5 of \cite{helm2007maps}, for instance), and similarly for $q$. Since passing from the ordinary tensor product to its quotient modulo its $\mathbb Z$-torsion can only decrease the dimension after reduction modulo $\mathfrak m$, applying this to formula (6) above, we obtain
		
		\begin{equation}
			\text{dim}~ [J^D(M)]_{\mathfrak{m}}/\mathfrak{m}[J^D(M)]_{\mathfrak{m}}\leq 2^2 ~\text{dim}~ [J^{D/pq}(pqM)]_{\tilde{\mathfrak{m}}}/\tilde{\mathfrak{m}}[J^{D/pq}(pqM)]_{\tilde{\mathfrak{m}}}.
		\end{equation}
		
		Now, $J^{D/pq}(pqM)$ in formula (7) above has discriminant $D/pq$ and we can repeat the same argument to bound $\text{dim}~ [J^{D/pq}(pqM)]_{\tilde{\mathfrak{m}}}/\tilde{\mathfrak{m}}[J^{D/pq}(pqM)]_{\tilde{\mathfrak{m}}}$. Hence, inductively, we obtain 
		
		\[
		\text{dim}~ [J^D(M)]_{\mathfrak{m}}/\mathfrak{m}[J^D(M)]_{\mathfrak{m}}\leq 2^k \text{dim}~ [J^{1}_{rs}(DM)]_{\mathfrak{M}}/\mathfrak{M}[J^{1}_{rs}(DM)]_{\mathfrak{M}}.
		\]
		The inequality $\text{dim}~J^1_{rs}(DM)[\mathfrak{M}]\leq \text{dim}~J^1(DM)[\mathfrak{M'}]$ follows immediately as $J^1_{rs}(DM)$ is a subvariety of $J^1(DM)$.
		\item
		Let $\mathfrak m$ be the preimage of $\mathfrak n$ under the natural quotient map
		\[
		\mathbb T^D(M)\rightarrow \mathbb T^D_Q(M).
		\]
		Tensoring formula~(6) of part~(1) over
		$\mathbb T^{D/pq}(pqM)_{\tilde{\mathfrak m}}$
		with $\mathbb T^D_Q(M)_{\mathfrak n}$, and using Lemma~8.2 and Corollary~5.3 of \cite{helm2007maps}, we obtain
		\[
		[J^D_Q(M)]_{\mathfrak n}
		\cong
		\Bigl(
		\mathcal X_p(J^{\min}_{DQ}(M/Q))_{\mathfrak n}
		\otimes
		\mathcal X_q(J^{\min}_{DQ}(M/Q))_{\mathfrak n}
		\otimes
		[J^{D/pq}_{Q}(pqM)]_{\tilde{\mathfrak m}}
		\Bigr)_{\mathrm{tf}}.
		\]
		The remainder of the proof is entirely analogous to that of part~(1), and we leave the details to the reader.
	\end{enumerate}
\end{proof}

We now prove Corollary 5.6.

\begin{proof}[Proof of Corollary 5.6]
	
	\begin{enumerate}
		\item Let $\mathfrak{M}$ and ${\mathfrak{M'}}$ be the preimages of $\mathfrak{m}$ under natural quotient maps $\mathbb{T}^1_{rs}(DM) \rightarrow \mathbb{T}^1_{D}(DM)\cong\mathbb{T}^D(M)$ and $\mathbb{T}^1(DM) \rightarrow \mathbb{T}^1_{D}(DM)\cong\mathbb{T}^D(M)$, respectively. One has $$\text{dim}~ J^D(M)[\mathfrak{m}]=\text{dim}~[J^D(M)]/\mathfrak{m}[J^D(M)]$$ and $$\text{dim}~ J^1_{rs}(DM)[\mathfrak{M}]=\text{dim}~[J^1_{rs}(DM)]/\mathfrak{M}[J^1_{rs}(DM)]$$ (see Corollary 8.10 of \cite{helm2007maps}). Proposition 5.5 gives
		\[
		\text{dim}~[J^D(M)]/\mathfrak{m}[J^D(M)]
		\leq
		2^k \text{dim}~[J^1_{rs}(DM)]/\mathfrak{M}[J^1_{rs}(DM)].
		\]
		By Theorem 5.1 (this is where we use the assumption that $l$ is odd and it is prime to $DM$),
		\[
		\text{dim}~J^{1}_{rs}(DM)[\mathfrak{M}]=2.
		\]
		Hence
		\[
		\text{dim}~J^D(M)[\mathfrak m]\leq 2^{k+1}.
		\]
		
		\item The proof of this part is essentially the same as the proof of the first part. Proposition 5.5 gives
		\[
		\text{dim}~J^D_Q(M)[\mathfrak n]
		\leq
		2^k\text{dim}~J^1_{rsQ}(DM)[\tilde{\mathfrak n}],
		\]
		and Theorem 5.1 gives
		\[
		\text{dim}~J^1_{rsQ}(DM)[\tilde{\mathfrak n}]=2.
		\]
		Therefore
		\[
		\text{dim}~J^D_Q(M)[\mathfrak n]\leq 2^{k+1}.
		\]
	\end{enumerate}
\end{proof}

Before ending this section, let
\[
\Phi:J^D(M)\rightarrow J^1_D(N)
\]
be a $\mathbb{T}^D(M)$-equivariant isogeny as in the introduction. Here, we use an \'etaleness condition on the kernel of the induced map $\tilde{\Phi}: \mathcal{J}^D(M) \rightarrow {\mathcal{J}^1_D(N)}$ on the N\'eron models at $l$, for any prime number $l$, to obtain the following multiplicity one result. Let $\Phi^{\vee}$ be the dual of $\Phi$ in the category of abelian varieties over $\mathbb{Q}$, and let $\tilde{\Phi}$ be the induced map on the N\'eron models at $l$.

\begin{prop}
	Let $\mathfrak{m}$ be the maximal ideal of $\mathbb{T}^D(M)$ defining the $\Gal(\overline{\mathbb{Q}}/\mathbb{Q})$-representation $\bar{\rho}(\mathfrak{m})$ attached to $A^D(M)[l]$. Assume, moreover, that $\mathfrak{m}$ is non-Eisenstein, and that $l$ is odd.
	\begin{enumerate}[label=(\alph*),start=\intfromalph{a}]
		\item 
		
		If the kernel of $\tilde{\Phi}^{\vee}$ is locally \'etale at $\mathfrak{m}$, then, $\text{Tan}_{\mathbb{Z}_l}(\mathcal{J}^D(M))_{\mathfrak{m}}\cong \mathbb{T}^D(M)_{\mathfrak{m}}$;
		
		\item 
		If $\text{Tan}_{\mathbb{Z}_l}(\mathcal{J}^D(M))_{\mathfrak{m}}\cong \mathbb{T}^D(M)_{\mathfrak{m}}$, and $l$ does not divide $DM$, then
		dim~$J^D(M)(\bar{\mathbb{Q}})[\mathfrak{m}]=2$;
		
		\item 
		If $\text{Tan}_{\mathbb{Z}_l}(\mathcal{J}^D(M))_{\mathfrak{m}}\cong \mathbb{T}^D(M)_{\mathfrak{m}}$, $l$ divides $M$, and $\mathfrak{m}$ is not $l$-old, then dim~$J^D(M)[\mathfrak{m}]=2$;
		
		\item 
		If $\text{Tan}_{\mathbb{Z}_l}(\mathcal{J}^D(M))_{\mathfrak{m}}\cong \mathbb{T}^D(M)_{\mathfrak{m}}$, and if $l$ divides $D$, then dim~$J^D(M)[\mathfrak{m}]=2$.
	\end{enumerate}
\end{prop}

\begin{proof}
	\begin{enumerate}[label=(\alph*),start=\intfromalph{a}]

		\item 
		In the category of group schemes over $\mathbb{Z}_l$, the morphism $\tilde{\Phi}$ on N\'eron models induces an exact sequence of free $\mathbb{Z}_l$-modules
		$$ 
		\text{Cot}_{\mathbb{Z}_l}(\mathcal{J}^D(M))
		\rightarrow 
		\text{Cot}_{\mathbb{Z}_l}({{{\mathcal{J}^1_D(N)}}}^{\vee})\rightarrow 	\text{Cot}_{\mathbb{Z}_l}(\mathrm{Ker}(\tilde{\Phi}^{\vee})).
		$$
		If $\tilde{\Phi}^{\vee}_{\mathfrak{m}}$ denotes the map induced from $\Phi$ by base changing to $\mathbb{T}^D(M)_{\mathfrak{m}}$, it follows from basic properties of group schemes that $\mathrm{Ker}(\tilde{\Phi}^{\vee}_{\mathfrak{m}})= \mathrm{Ker}(\tilde{\Phi}^{\vee}) \times_{\mathrm{Spec}(\mathbb{Z}_p)}\mathrm{Spec}(\mathbb{T}^D(M)_{\mathfrak{m}})$. By the \'etale assumption at $\mathfrak{m}$, $\text{Cot}_{\mathbb{F}_l}(\mathrm{Ker}(\tilde{\Phi}^{\vee}_{\mathfrak{m}}))=0$. But, by the base change compatibility of the cotangent functor, $\text{Cot}_{\mathbb{F}_l}(\mathrm{Ker}(\tilde{\Phi}^{\vee}))_{\mathfrak{m}}=\text{Cot}_{\mathbb{F}_l}(\mathrm{Ker}(\tilde{\Phi}^{\vee}_{\mathfrak{m}}))=0$. Thus, passing to special fibers, we obtain 
		\begin{equation}
			\text{Cot}_{\mathbb{F}_l}(\mathcal{J}^D(M))_{\mathfrak{m}}
			\rightarrow 
			\text{Cot}_{\mathbb{F}_l}({{\mathcal{J}^1_D(N)}}^{\vee})_{\mathfrak{m}}\rightarrow 0. 
		\end{equation}
		Since N\'eron models and their generic fibers have the same dimension, and $J^D(M)$ and $J^1_D(M)^{\vee}$ have equal dimensions, we conclude that $$\mathrm{dim}_{\mathbb{F}_l}(	\text{Cot}_{\mathbb{F}_l}(\mathcal{J}^1_D(N)^{\vee}))= \mathrm{dim}_{\mathbb{F}_l}(	\text{Cot}_{\mathbb{F}_l}({{\mathcal{J}}^D(M)})).$$ Hence, the surjection in equation (8) is an isomorphism. In particular, $\text{Tan}_{\mathbb{Z}_l}(\mathcal{J}^D(M))_{\mathfrak{m}}\cong \text{Tan}_{\mathbb{Z}_l}({\mathcal{J}^1_D(N)}^{\vee})_{\mathfrak{m}}$. Now, by Corollary 1.1 in \cite{mazur1978rational},
		
		the exact sequence 
		\[
		0 \rightarrow J^1(N)^{D-\mathrm{old}} \rightarrow J^1(N) \rightarrow {J^1_D(N)}^{\vee} \rightarrow 0
		\]
		of abelian varieties over $\mathbb{Q}_l$ induces an injection 
		$$ 
		\text{Cot}_{\mathbb{F}_l}({{\mathcal{J}}^1_D(N)}^{\vee})[\mathfrak{m}] \hookrightarrow \text{Cot}_{\mathbb{F}_l}({\mathcal{J}}^1(N))[\mathfrak{\tilde{m}}],
		$$
		where $\mathfrak{m}$ is the maximal ideal of $\mathbb{T}^D(M)$ attached to $E[l]$ and $\mathfrak{\tilde{m}}$ is its lifting to $\mathbb{T}^1(N)$. Since $\text{Cot}_{\mathbb{F}_l}({\mathcal{J}}^1(N))[\mathfrak{\tilde{m}}]$ is isomorphic to $H^0(X_0(N)_{/\mathbb{F}_l},\Omega^1)[\mathfrak{\tilde{m}}]$, and the latter is one-dimensional over $\mathbb{T}^1(N)/\mathfrak{\tilde{m}}$ by \cite{Agashe2012}, we conclude that $\text{Cot}_{\mathbb{F}_l}({\mathcal{J}}^1(N))[\mathfrak{\tilde{m}}]$ is of dimension one over $\mathbb{T}^D(M)/\mathfrak{m}$. Therefore, by duality, $\text{Tan}_{\mathbb{Z}_l}({\mathcal{J}^1_{D}(N)}^{\vee})_{\mathfrak{m}}\cong \mathbb{T}^D(M)_{\mathfrak{m}}$.

		\item 
		It is a standard fact that $ \text{Tan}_{\mathbb{Z}_l}(\mathcal{J}^D(M))$ is isomorphic to $\text{H}^1(\mathcal{J}^D(M),\mathcal{O}_{\mathbb{Z}_l})$. By part (a), and the argument in part (b) of Theorem 5.2 of \cite{ribet1990modular}, which uses Dieudonn\'e theory, the result follows.
		\item 
		The approach used in \cite{mazur1991two} can be naturally generalized to this setting using Buzzard's model over $\mathbb{Z}_l$ \cite{buzzard1997integral} of the Shimura curve $X^D(M)$ attached to a rational indefinite quaternion algebra of discriminant $D$, whose Jacobian, in this paper, is denoted by $J^D(M)$. The statement follows in a manner similar to Proposition 22 of \cite{mazur1991two}.
		\item 
		For this part, we use the admissible model over $\mathbb{Z}_l$ for $X^D(M)$ constructed by Cerednik-Drinfeld in \cite{vcerednik1976uniformization} and \cite{drinfel1976coverings}. Since $J^D(M)$ has purely toric reduction at $l$, the non-$l$-old condition of Proposition 22 of \cite{mazur1991two} is automatic. One can then adapt Proposition 22 of $loc.~ cit.$ to conclude Jacobian multiplicity one.
	\end{enumerate}
\end{proof}
\begin{cor}
	With conditions as in Proposition 5.8, if $\mathrm{ord}_l(\delta^D(M))< \mathrm{ord}_l(\delta^1_D(N))$, then for any Hecke-equivariant isogeny $\Phi$ as in the proposition, $\tilde{\Phi}^{\vee}$ is not \'etale at $\mathfrak{m}$. 
\end{cor}
\begin{proof}
	The corollary follows from Proposition 3.4, Theorem A, and Proposition 5.8. 
\end{proof}

%% file: Proof_of_B.tex
\section{Proof of Theorem B}

Finally, we put everything together to prove Theorem B.

\begin{proof}[Proof of Theorem B]
	First, assume that $\mathfrak{m}$ is non-Eisenstein.
	Let $\tilde{\mathfrak{m}}$ be the maximal ideal of $\mathbb{T}^1(N)$ which is the preimage of $\mathfrak{m}$ under the natural quotient map $\mathbb{T}^1(N) \rightarrow \mathbb{T}^1_D(N) \cong \mathbb{T}^D(M)$. By Theorem 5.1, part (a), if $l \nmid N$, $\dim J^1(N)[\tilde{\mathfrak{m}}]=2$. Hence, $\mathrm{dim}\, J^1_D(N)[\mathfrak{m}]=2$.
	If $l \mid N$, it is known that $a_l=1 \text{ or } -1$, and that, up to an unramified quadratic twist,
	\[
	\bar{\rho}_{\mathfrak{m}}|_{G_{\mathbb{Q}_l}} \cong
	\begin{pmatrix}
		\overline{\chi}_l & * \\
		0 & 1
	\end{pmatrix},
	\]
	where $\overline{\chi}_l$ is the mod-$l$ cyclotomic character. Since $l$ is odd, $\overline{\chi}_l \neq 1$, i.e. $\tilde{\mathfrak{m}}$ is $G_l$-distinguished in the sense of \cite{tilouine1997hecke}. Therefore, by part (b) of Theorem 5.1, it again follows that $\mathrm{dim}\, J^1_D(N)[\mathfrak{m}]=2$.
	By Proposition 3.4 and Theorem A, $\mathrm{ord}_l(r^1_D(N))=\mathrm{ord}_l(\delta^1_D(N)) \geq \mathrm{ord}_l(\delta^D(M))$.
	On the other hand, if $\mathfrak{m}$ is controllable at all primes dividing $D$, then by Proposition 5.5 and the previous lines, we have $\dim J^D(M)[\mathfrak{m}]=2$.
	Using Proposition 3.4 and Theorem A for elliptic curves $A^D(M)$ and $A^1_D(N)$, we conclude that
	\[
	\mathrm{ord}_l(\delta^1_D(N))=\mathrm{ord}_l(r^1_D(N))= \mathrm{ord}_l(r^D(M))= \mathrm{ord}_l(\delta^D(M)).
	\]
	If $\mathfrak{m}$ is not controllable at exactly one prime $p$ dividing $D$, then by Proposition 5.5, if $q$ is any other prime dividing $D$, $\dim J^{\frac{D}{pq}}(Mpq)[\mathfrak{m}]=2$.
	Therefore,
	\[
	\mathcal{X}_p(J^D(M))^*_{\mathfrak{m}} \cong 
	\mathcal{X}_q(J^{\frac{D}{pq}}(Mpq))_{\mathfrak{m}} \otimes_{\mathbb{T}_{\mathfrak{m}}^{\frac{D}{pq}}} \mathbb{T}^D(M)_{\mathfrak{m}}
	\]
	(\cite{helm2007maps}, Proposition 5.8)
	\[
	\cong \mathbb{T}^D(M)_{\mathfrak{m}}
	\]
	(\cite{helm2007maps}, Lemma 6.5 and the controllability assumption of $\mathfrak{m}$ at $q$).
	
	Since $J^D(M)$ has purely toric reduction at the prime $p$, the action of $\mathbb{T}^1_{D}(N)$ and its saturation $\mathbb{T}'^1_{D}(N)$ on $\mathcal{X}_p(J^D(M))^*$ is faithful. Hence,
	\[
	\mathbb{T}^1_{D}(N)'_{\mathfrak{m}}=\mathbb{T}^1_{D}(N)_{\mathfrak{m}}.
	\]
	Applying Proposition 3.4 and Theorem A once again, we obtain
	\[
	\mathrm{ord}_l(\delta^1_D(N))=\mathrm{ord}_l(r^1_D(N))= \mathrm{ord}_l(r^D(M))= \mathrm{ord}_l(\delta^D(M)).
	\]
	
     It remains to prove that if $\mathfrak{m}$ is Eisenstein, then
	\[
	\mathrm{ord}_l(r^1_D(N))=\mathrm{ord}_l(\delta^1_D(N))\geq \mathrm{ord}_l(\delta^D(M)).
	\]
	We proved (without the non-Eisenstein assumption on $\mathfrak{m}$) in Proposition 5.8, part (a), that
	\[
	\text{Tan}_{\mathbb{Z}_l}({{J}^1_{D}(N)}^{\vee})_{\mathfrak{m}}\cong \mathbb{T}^D(M)_{\mathfrak{m}}.
	\]
	Hence,
	\[
	\mathbb{T}^D(M)'_{\mathfrak{m}} \cong \mathbb{T}^D(M)_{\mathfrak{m}}.
	\]
	Note that the Hecke algebra $\mathbb{T}^D(M)$ acting on ${J}^1_D(N)^{\vee}$ is the image of the Hecke algebra $\mathbb{T}^D(M)$ acting on $J^1_D(N)$ under the natural isomorphism
	\[
	\mathrm{End}(J^1_D(N)) \rightarrow \mathrm{End}({J}^1_D(N)^{\vee}).
	\]
	Consequently, the same isomorphism
	\[
	\mathbb{T}^D(M)'_{\mathfrak{m}} \cong \mathbb{T}^D(M)_{\mathfrak{m}}
	\]
	holds when viewed in $\mathrm{End}(J^1_D(N))$.
	This implies that
	\[
	\mathrm{ord}_l(\delta^1_D(N)) = \mathrm{ord}_l(r^1_{D}(N)).
	\]
	
\end{proof}

%% file: Case_of_l=2_and_a_Criterion.tex
\section{$\ell=2$ and a Criterion}

The case $l=2$ is more subtle because sufficiently strong multiplicity one results are not available in this setting. To our knowledge, the strongest result in this case is Theorem 5.1, part (d).

Now, let $E$ be a rational elliptic curve of conductor $N=2M$, where $M$ is square-free and odd. If $E[2]$ is irreducible as a $G_{\mathbb{Q}}$-module, then for every divisor $Q$ of $M$, by Theorem 5.1, part (c), and Proposition 3.4,
\[
\mathrm{ord}_2(\delta^1_{2Q}(2M))=\mathrm{ord}_2(r^1_{2Q}(2M)).
\]
In the following proposition, we remove the non-$2$-old condition.

\begin{prop}
	Let $E$ be a rational elliptic curve of conductor $N=2M$ for a square-free and odd natural number $M$. If $E[2]$ is irreducible as a $G_{\mathbb{Q}}$-module, then for any $Q$ dividing $M$,
	\[
	\mathrm{ord}_2(\delta^1_{2Q}(N))=\mathrm{ord}_2(r^1_{2Q}(N)).
	\]
	In particular,
	\[
	\mathrm{ord}_2(\delta^{2Q}(N/2Q)) \leq \mathrm{ord}_2(\delta^1_{2Q}(N))
	\]
	for any $Q$ dividing $M$ which is a product of an odd number of primes.
\end{prop}

\begin{proof}
	It is known that
	\[
	\mathrm{Hom}(\mathcal{X}_2(J^1(N))/\mathfrak{m}\mathcal{X}_2(J^1(N)), \mathbb{F}_2)
	\]
	is isomorphic, as a $\mathbb{T}^1(N)$-module, to
	\[
	H^0\!\left(X_0(N)_{/\mathbb{F}_2}, \Omega^1\right)[\mathfrak{m}]
	\]
	(see Proposition 18 of \cite{mazur1991two}, for instance). The latter has dimension $1$, and therefore by Nakayama's lemma, there is a surjection
	\[
	\mathbb{T}^1(N)_{\mathfrak{m}} \rightarrow \mathcal{X}_2(J^1(N))_{\mathfrak{m}}
	\]
	of $\mathbb{T}^1(N)$-modules. By Theorem 3.10 of \cite{ribet1990modular}, it induces an isomorphism
	\[
	\mathcal{X}_2(J^1(N))_{\mathfrak{m}} \cong \mathbb{T}^1_2(N)_{\mathfrak{m}}.
	\]
	By Lemma 5.1 of \cite{helm2007maps}, one has
	\[
	\mathcal{X}_2(J^1_{2Q}(N))_{\mathfrak{m}}
	\cong
	\mathcal{X}_2(J^1(N))_{\mathfrak{m}}
	\otimes_{\mathbb{T}^1(N)_{\mathfrak{m}}}
	\mathbb{T}^1_{2Q}(N)_{\mathfrak{m}}.
	\]
	By Proposition 3.4,
	\[
	\mathrm{ord}_2(\delta^1_{2Q}(N))=\mathrm{ord}_2(r^1_{2Q}(N)).
	\]
	The second part follows from the first part and Theorem A.
\end{proof}

It is not known whether Proposition 7.1 remains valid when $E[2]$ is reducible or when $N$ is odd. In this direction, the following proposition predicts the failure of multiplicity one in characteristic $2$ and could potentially produce examples of failure of multiplicity one similar to those in \cite{kilford2002some}.

\begin{prop}[Criterion]
	Suppose $\mathfrak{m}$ is the maximal ideal of $\mathbb{T}^1_{D}(N)$ attached to $E[2]$. If
	\[
	\mathrm{ord}_2(\delta^D(M))>\mathrm{ord}_2(\delta^1_{D}(N)),
	\]
	then
	\begin{enumerate}[label=(\alph*),start=\intfromalph{a}]
		\item $\dim J^1_{D}(N)[\mathfrak{m}]>2$. In particular, $\dim J^1(N)[\mathfrak{m}]>2$;
		\item $\dim \mathrm{Cot}_{\mathbb{F}_2}(\mathcal{J}^1_{D}(N))[\mathfrak{m}] >1$;
		\item $\dim \mathcal{X}_p(J^1_{D}(N))/\mathfrak{m}\mathcal{X}_p(J^1_{D}(N))>1$ for any prime $p$ dividing $D$.
	\end{enumerate}
\end{prop}

\begin{proof}
	This follows immediately from Proposition~3.4.
\end{proof}

%% file: Higher_dim.tex
\section{Higher Dimensions}

The results of previous sections can be generalized to modular abelian varieties attached to newforms. In this section, as an example, we generalize Corollary 4.1 (Proposition 8.2) and Proposition 7.2 (Proposition 8.3). The latter will be used in Section 9.1. Let $N=DM$ be square-free, where $D$ is a product of an even number of primes. Let $f$ be a newform of level $N$ and weight 2. Similarly to the case of elliptic curves, one defines $\mathbb{T}^D_Q(M)$-invariant abelian varieties $A^D_Q(M)$ and $B^D_Q(M)$ such that
\[
J^D_Q(M)=A^D_Q(M)+B^D_Q(M).
\]
Let $\mathcal{O}_f$ be the $\mathbb{Z}$-algebra generated by Fourier coefficients of $f$. Let
\[
I=\ker\bigl(\pi_{A^D_Q(M)}:\mathbb{T}^D_Q(M)\to \mathcal{O}_f \hookrightarrow \mathrm{End}(A^D_Q(M))\bigr).
\]
Similarly, one has a map
\[
\pi_{B^D_Q(M)}:\mathbb{T}^D_Q(M)\to \mathrm{End}(B^D_Q(M)),
\]
and let $\mathbb{T}^D_Q(M)_{A^D_Q(M)}$ and $\mathbb{T}^D_Q(M)_{B^D_Q(M)}$ be
\[
\pi_{A^D_Q(M)}(\mathbb{T}^D_Q(M))=\mathcal{O}_f,
\qquad
\pi_{B^D_Q(M)}(\mathbb{T}^D_Q(M)).
\]

The following definition generalizes Definition 3.1.

\begin{Def}
	\begin{enumerate}
		\renewcommand{\labelenumi}{(\roman{enumi})}
		\item The \emph{congruence ideal} attached to $A^D_Q(M)$ is
		\[
		R^D_Q(M):=\pi_{A^D_Q(M)}\bigl(\ker(\pi_{B^D_Q(M)})\bigr)\subseteq \mathcal{O}_f.
		\]
		
		The \emph{congruence exponent}, respectively the \emph{congruence number}, attached to $A^D_Q(M)$ are
		\[
		r^D_Q(M)=\exp\bigl(\mathcal{O}_f/R^D_Q(M)\bigr),
		\qquad
		\tilde r^D_Q(M)=\#\bigl(\mathcal{O}_f/R^D_Q(M)\bigr);
		\]
		
	\item	The \emph{intersection ideal} attached to $A^D_Q(M)$ is
		\[
		S^D_Q(M):=\operatorname{Ann}_{\mathcal{O}_f}\bigl(A^D_Q(M)\cap B^D_Q(M)\bigr).
		\]
		
		The \emph{modular exponent}, respectively the \emph{modular number}, attached to $A^D_Q(M)$ are
		\[
		\delta^D_Q(M)=\exp\bigl(A^D_Q(M)\cap B^D_Q(M)\bigr),
		\qquad
		\tilde{\delta}^D_Q(M)=\#\bigl(A^D_Q(M)\cap B^D_Q(M)\bigr).
		\]
		
	\end{enumerate}
\end{Def}

The following proposition generalizes Corollary 4.1:

\begin{prop}
	\[
	\delta^D(M)\mid \delta^1(N).
	\]
\end{prop}

\begin{proof}
	Let $\Phi$ be the $\mathbb{T}^D(M)$-equivariant isogeny in the beginning of the proof of Theorem A. One obtains a similar decomposition for $J^1_D(N)$ of the form
	\[
	\Phi(A^D(M))+\Phi(B^D(M)),
	\]
	and one proves that $r^D(M)=r^1_D(N)$, where $r^1_D(N)$ is the $D$-new congruence number attached to $A$. Since $f$ is a newform, it is proven in \cite{agashe2008modular} that
	\[
	\mathrm{ord}_p(r^1(N))=\mathrm{ord}_p(\delta^1(N)).
	\]
 It is not hard to see that
 \[
 \exp\bigl((A^D(M)\cap B^D(M))_{p}\bigr)
 =
 \exp\bigl((\mathcal O_f/S^D(M))_{p}\bigr).
 \]
 Moreover, from Definition 8.1,
 \[
 \pi_{A^D(M)}(\ker(\pi_{B^D(M)}))
 \subseteq
 \operatorname{Ann}_{\mathcal O_f}(A^D(M)\cap B^D(M)).
 \]
 Hence there is a canonical quotient map
 \[
 \mathcal O_f/R^D(M)\twoheadrightarrow \mathcal O_f/S^D(M).
 \]
 Therefore
 \[
 \exp\bigl((\mathcal O_f/S^D(M))_p\bigr)
 \le
 \exp\bigl((\mathcal O_f/R^D(M))_p\bigr).
 \]
 Taking $p$-adic valuations yields
 \[
 \operatorname{ord}_p(\delta^D(M))
 \le
 \operatorname{ord}_p(r^D(M)).
 \]
 It follows easily from the definition that $r^1_D(N) \mid r^1(N)$ and hence
 
\[
 \operatorname{ord}_p(\delta^D(M)) \le
 \operatorname{ord}_p(r^D(M)) =
 \operatorname{ord}_p(r^1_D(N))\le 
 \operatorname{ord}_p(r^1(N))=
 \operatorname{ord}_p(\delta^1(N)).
\]
\end{proof}

Here, we state the analog of Proposition 7.2 and add another part, which is a generalization of Theorem 2.1 of \cite{agashe2008modular}.

\begin{prop}
	With notation as above:
	
	\medskip
	\noindent
	1. If $\mathrm{ord}_2(r^D_{Q}(M))>\mathrm{ord}_2(\delta^D_{Q}(M))$, then there is a maximal ideal $\mathfrak{m}$ of $\mathbb{T}^D_{Q}(M)$ of characteristic 2, containing $I_f$, such that
	\begin{enumerate}[label=(\alph*), start=\intfromalph{a}]
		\item $\dim J^D_{Q}(M)[\mathfrak{m}]>2$. In particular, $\dim J^D(M)[\mathfrak{m}]>2$;
		\item $\dim \mathrm{Cot}_{\mathbb{F}_2}(\mathcal{J}^D_{Q}(M))[\mathfrak{m}]>1$;
		\item $\dim \mathcal{X}_p(J^D_{Q}(M))/\mathfrak{m}\mathcal{X}_p(J^D_{Q}(M))>1$ for any prime $p$ dividing $D$.
	\end{enumerate}
	
	\medskip
	\noindent
	2. If $\mathrm{ord}_2(\tilde{\delta}^D_{Q}(M)) \neq 2\,\mathrm{ord}_2(\tilde{r}^D_{Q}(M))$, then there exists a maximal ideal $\mathfrak{m}$ of $\mathbb{T}^D_{Q}(M)$ of characteristic 2, containing $I_f+\operatorname{Ann}_{\mathbb{T}^D_{Q}(M)}(I_f)$, such that
	\[
	\dim J^D_{Q}(M)^{\vee}[\mathfrak{m}]>2.
	\]
	In particular, if $\tilde{\delta}^D_{Q}(M)$ is not a perfect square, then there exists a maximal ideal $\mathfrak{m}$ of $\mathbb{T}^D_{Q}(M)$ containing $I_f+\operatorname{Ann}_{\mathbb{T}^D_{Q}(M)}(I_f)$ such that
	\[
	\dim J^D_{Q}(M)^{\vee}[\mathfrak{m}]>2.
	\]
	
\end{prop}

\begin{proof}
	The proof of the first part is similar to the proof of Proposition 7.2, and we omit it.
	
	For part 2, assume that $\mathfrak{m}$ is a maximal ideal of $\mathbb{T}^D_{Q}(M)$ containing $I_f+\operatorname{Ann}_{\mathbb{T}^D_{Q}(M)}(I_f)$ for which $\dim J^D_{Q}(M)^{\vee}[\mathfrak{m}]=2$. Arguing similarly to Theorem 2.2 of \cite{agashe2008modular},
	\[
	A^D_{Q}(M)\cap B^D_{Q}(M)\cong
	\frac{H_1(J^D_{Q}(M),\mathbb{Z})}
	{H_1(J^D_{Q}(M),\mathbb{Z})[I_f]+H_1(J^D_{Q}(M),\mathbb{Z})[\operatorname{Ann}_{\mathbb{T}^D_{Q}(M)}(I_f)]}.
	\]
	By duality, there is a natural $\mathbb{T}^D_{Q}(M)$-module isomorphism between $H_1(J^D_{Q}(M),\mathbb{Z})$ and $H^1(J^D_{Q}(M)^{\vee},\mathbb{Z})$. Hence,
	
\begin{equation}
	A^D_{Q}(M)\cap B^D_{Q}(M)\cong
	\frac{H^1(J^D_{Q}(M)^{\vee},\mathbb{Z})}
	{H^1(J^D_{Q}(M)^{\vee},\mathbb{Z})[I_f]+H^1(J^D_{Q}(M)^{\vee},\mathbb{Z})[\operatorname{Ann}_{\mathbb{T}^D_{Q}(M)}(I_f)]}.
\end{equation}
	
	Since $\dim J^D_{Q}(M)^{\vee}[\mathfrak{m}]=2$, and $J^D_{Q}(M)^{\vee}[\mathfrak{m}]$ is isomorphic to
	 $$\text{Hom}(	H^1(J^D_{Q}(M)^{\vee},\mathbb{Z})/\mathfrak{m}H^1(J^D_{Q}(M)^{\vee},\mathbb{Z}), ~ \mathbb{T}^D_{Q}(M)/\mathfrak{m}\mathbb{T}^D_{Q}(M)),$$
	 it follows that $\dim_{\mathbb{T}^D_{Q}(M)/\mathfrak{m}\mathbb{T}^D_{Q}(M)} H^1(J^D_{Q}(M)^{\vee},\mathbb{Z})/\mathfrak{m}H^1(J^D_{Q}(M)^{\vee},\mathbb{Z})=2.$
	 
	 Therefore, applying Nakayama's lemma, we obtain a surjection
	 
	\begin{equation}
		    \mathbb{T}^D_{Q}(M)_{\mathfrak{m}}^2 \twoheadrightarrow H^1(J^D_{Q}(M)^{\vee},\mathbb{Z})_{\mathfrak{m}}.
	\end{equation}
 By Jacquet-Langlands and Eichler-Shimura, surjection (10)  induces an isomorphism after tensoring with $\mathbb{Q}_2$. But $\mathbb{T}^D_{Q}(M)_{\mathfrak{m}}^2$ is flat over $\mathbb{Z}_2$, and therefore (10) is an isomorphism. Combining isomorphisms (9) and (10) above,
	\[
	A^D_{Q}(M)\cap B^D_{Q}(M)\cong
	\Bigl(\frac{\mathbb{T}^D_{Q}(M)}{I_f+\operatorname{Ann}_{\mathbb{T}^D_{Q}(M)}(I_f)}\Bigr)_{\mathfrak{m}}^2.
	\]
    Hence,
	\[
	\mathrm{ord}_2(\tilde{\delta}^D_{Q}(M))=2\,\mathrm{ord}_2(\tilde{r}^D_{Q}(M)).
	\]
\end{proof}

\begin{rem}
	We have focused on the prime 2 in Propositions 7.2 and 8.3, since multiplicity one rarely fails at odd primes in the setting considered here.
\end{rem}

%% file: Applications.tex
\section{Applications}

\subsection{Failure of Multiplicity One}

Using Sage \cite{stein2007sage}, in the following table we have listed some modular abelian varieties $A_D(M) \subseteq J^1_D(M)$ attached to the newforms of level less than 300 for which $\mathrm{ord}_2(\tilde{\delta}^1_D(N)) \neq 2~\mathrm{ord}_2(\tilde{r}^1_{D}(N))$. Hence, by Proposition 8.3, they produce examples of the failure of the Jacobian multiplicity one. For a fixed conductor, there might be more than one such abelian varieties, or there might be different values of $D$ for an abelian variety such that $\mathrm{ord}_2(\tilde{\delta}^1_D(N)) \neq 2~\mathrm{ord}_2(\tilde{r}^1_{D}(N))$, but we list only one example for each level as follows.
\begin{table}[htbp]
	\centering
	\caption{Producing $\mathrm{dim}~{{J}^1_{D}(N)}^{\vee}[\mathfrak{m}]>2$}
	\label{tab:mytable}
	\begin{tabular}{c c c c c c}
		\hline
		$N$ & $D$ & Isogeny class of $A^1_D(N)$ (Sage label) & Dimension of $A^1_D(N)$ &$\tilde{\delta}^1_D(N)$ & $\tilde{r}^1_{D}(N)$\\
		\hline
		105 & 3 & 105a & 1 & 8 & 4 \\
		165 & 30 & 165a & 2 & 128 & 16\\
		195 & 195 & 195e & 3 & 32 & 8 \\
		219 & 219 & 219e & 6 & 2048 & 32 \\
		238 & 7 & 238a & 1 & 32 & 8\\
		273 & 21 & 273c & 2 & 32 & 8 \\
		285 & 5 & 285d & 2 & 14971392 & 5472 \\
		291 & 3 & 291h & 7 & 32768 & 128 \\
		\hline
	\end{tabular}
\end{table}

Numerically, it appears that the failure of Jacobian multiplicity one occurs frequently at maximal ideals of characteristic 2. On the other hand, for any $D$ dividing a square-free level $N$, for $N$ up to 500, the $D$-new modular exponent $\delta^1_D(N)$ is equal to the $D$-new congruence exponent $r^1_D(N)$, hence supporting Conjecture 5.1.7 in \cite{deines2014shimura}.

Now, suppose $N=DM$ is as before, and let $A^D(M)$ be any modular abelian subvariety of $J^D(M)$ attached to a newform of level $DM$. Let $\phi: \Phi_p(J^D(M)) \rightarrow \Phi_p(A^D(M)^{\vee})$ be the induced map on the groups of components of the Néron models at $p$ for a prime $p$ dividing $D$. Let $i_p$, $j_p$, and $h_p$ be the size of the image of $\phi$, the size of the cokernel of $\phi$, and $\langle g_p, g_p \rangle$, respectively, where $\langle \cdot,\cdot \rangle$ is the monodromy pairing on the character group $\mathcal{X}_p(J^D(M))$, and $g_p$ is a generator of $\Lambda^d(\mathcal{X}_f)$, the $d$-th exterior power of the $f$-isotypical component $\mathcal{X}_f$ of $\mathcal{X}_p(J^D(M))$ (see Section 2 of \cite{takahashi2001degrees} for detailed definitions).
 We have the following result:

\begin{thm}
	
	\begin{enumerate}
		\item If 
		$h_p/i_p > \tilde{r}^D(M)$, then there is a maximal ideal $\mathfrak{m}$ of $\mathbb{T}^D(M)$ containing $I_f$ such that $\dim J^D(M)[\mathfrak{m}]>2$;
		
		\item If $A$ is an elliptic curve, then $h_p/i_p \leq r^D(M)$ and, if $h_p/i_p = r^D(M)$, then $h_p/i_p = \delta^D(M)$. In particular, if $h_p/i_p = r^D(M)$ in this case, then $j_p=1$, i.e. $\phi: \Phi_p(J^D(M)) \rightarrow \Phi_p(A^D(M))$ is surjective.
	\end{enumerate}
	
\end{thm}

\begin{proof}
	
	\begin{enumerate}
		\item  It is known that $\sqrt{\tilde{\delta}^D(M)}=(h_p/i_p)\, j_p$ for the Shimura degree of the abelian variety $A^D(M)$ as a generalization of Theorem 2.3 in \cite{takahashi2001degrees}, which is for the case of elliptic curves (see Section 3.2 of \cite{deines2014shimura} for more details). Now, if $h_p/i_p > \tilde{r}^D(M)$, then $\sqrt{\tilde{\delta}^D(M)}>\tilde{r}^D(M)$. Hence, by Proposition 8.3, the result follows.
		
		\item In this case, $r^D(M)=\tilde{r}^D(M)$ and $\delta^D(M)^2=\tilde{\delta}^D(M)$. The theorem follows from the formula $\sqrt{\tilde{\delta}^D(M)} = (h_p/i_p)\, j_p$ and $\delta^D(M) \leq r^D(M)$.
	\end{enumerate}
	
\end{proof}

We have used Magma \cite{bosma1997magma} and an adaptation of the code in Appendix A.1.1 of \cite{deines2014shimura}, to higher dimensions, built on Brandt modules, to compute the quantities $h_p$ and $i_p$ attached to different simple optimal quotients of $J^D(M)$ for varying $D$ and $M$. In the following table, we have listed some cases where $\mathrm{ord}_2(h_p/i_p) > \mathrm{ord}_2(\tilde{r}^D(M))$, and hence, by the above theorem, they provide examples of the failure of multiplicity one for $J^D(M)$ in characteristic $2$.

\begin{table}[htbp]
	\centering
	\caption{Examples of $\mathrm{dim}~J^D(M)[\mathfrak{m}]>2$}
	\label{tab:mytable}
	\begin{tabular}{c c c c c c c}
		\hline
		$N$ & $D$ & $p$ & Isogeny class of $A^D(M)$ (Magma label) & Dimension of $A^D(M)$ & $h_p/i_p$ & $\tilde{r}^{D}(M)$\\
		\hline
		105 & 35 & 5 & 105B & 2 & $20$ & 10 \\
		138 & 69 & 23 & 138D & 2 & $16$ & 8 \\
		154 & 77 & 7 & 154D & 2 & $16$ & 8 \\
		165 & 55 & 11 & 165C & 3 & $32$ & 16 \\
		221 & 221 & 17 & 221G & 6 & $256$ & 128\\
		273 & 39 & 3 & 273E & 4 & $128$ & 64 \\
		282 & 141 & 47 & 282E & 3 & $320$ & 160\\
		310 & 155 & 5 & 310E & 3 & $320$ & 160\\
		322 & 161 & 7 & 322G & 3 & $1408$ & 704 \\
		\hline
	\end{tabular}
\end{table}

\subsection{On A Conjecture of Ribet and Takahashi}
\subsubsection{Counterexamples}
Let $N=DM$ be as before, and let $A^D(M)$ be the elliptic
curve defined in Definition~3.1, viewed as an abelian
subvariety of $J^D(M)$. Since both $J^D(M)$ and $A^D(M)$
are self-dual, we may also view $A^D(M)$ as an optimal
quotient of $J^D(M)$.

For a prime $p\mid D$, Ribet and Takahashi proved in
\cite{ribet1997parametrizations} that if $M$ is not prime
and the Galois representation on $A^D(M)[\ell]$ is
irreducible, then the $\ell$-primary part of the cokernel
of the map
$
\phi:\Phi_p(J^D(M))\longrightarrow\Phi_p(A^D(M)),
$
introduced in Section~9.1, is trivial. In subsequent work,
we remove the restriction that $M$ is not prime whenever
$\ell>2$ (see \cite[Theorem~5.9, Part~(1)]{hamidi2026hecke}).

The conjecture that $\phi$ is always surjective, or more
generally that the analogous map associated with an optimal
elliptic quotient of any Jacobian with purely toric reduction
at $p$ is surjective, is known as the Ribet--Takahashi
conjecture. In \cite{papikian2016optimal}, Papikian and
Rabinoff constructed counterexamples to the general
conjecture using certain Jacobians isogenous to a product of two
elliptic curves (see Section~5 of \emph{loc.~cit.}).
They also asked whether surjectivity holds for $J^D(M)$
when $p$ divides $D$. Here, we show that the answer to
this question is negative in general:
 \begin{thmC}
 	The map $ \phi: \Phi_p(J^D(M)) \rightarrow \Phi_p(A^D(M))$ is not necessarily surjective.
 \end{thmC}

\begin{proof}
	Recall that $\delta^D(M)= (h_p/i_p).j_p$ is independent of the prime $p$ dividing $DM$. Now, if $p$ and $q$ are two different primes dividing $D$ such that $h_p/i_p \neq h_q/i_q$, then it follows that either $j_p \neq 1$ or $j_q \neq 1$. The theorem follows from either case. In Magma, we have computed that for the elliptic curve 102c(1,102) with  $D=6~(\mathrm{and~ hence}~M=17)$, one finds that $h_2/i_2=2$ while $h_3/i_3=1$.
\end{proof}
	
	\begin{rem}
		The elliptic curve 102c(1,102) in the above theorem is the example of smallest conductor for which \(j_p\ne1\), in the sense that for any elliptic curve $E$ of a square-free conductor $N < 102$ and any $D$ dividing $N$ with an even number of prime factors, it is computed in Magma that $h_p/i_p = r^D(M)$ for all primes $p$ dividing $D$. Hence, by Theorem 9.1, part 2, $j_p=1$ for all such primes.
	\end{rem}
		The method of the above theorem can also be applied to the higher dimensional abelian varieties to obtain cases where the map on component groups is not surjective. In other words, one can compute the quantity $h_p/i_p$ attached to the higher dimensional abelian varieties for different primes $p$ dividing $D$. Using this method, we have obtained the following table of some modular abelian varieties $A^D(M)$ attached to the newforms such that $ \phi: \Phi_p(J^D(M)) \rightarrow \Phi_p(A^D(M)^{\vee})$ is not surjective.
\begin{table}[htbp]
	\centering
	\caption{$j_p \neq 1$}
	\label{tab:mytable}
	\begin{tabular}{c c c c c c c c}
		
		\hline
		
		$N$ & $D$ & Isogeny class of $A^D(M)$ (Magma label) & Dimension of $A^D(M)$&$p$ & $q$ & $h_p/i_p$ & $h_q/i_q$\\
		\hline
		102 & 6 & 102C & 1 & 2 & 3 & 2 & 1 \\
		210 & 14 & 210C & 1 & 2 & 7 & 4 & 2 \\
		273 & 21 & 273E & 4 & 3 & 7 & 64 & 32 \\
		221 & 221 & 221E & 2 & 13 & 17 & 40 & 80 \\
		322 & 161 & 322G & 3 & 7 & 23 & 1408 & 704 \\
		357 & 119 &  357H & 4 & 7 & 17 &  14272 & 28544\\	
		
		\hline
	\end{tabular}
\end{table}

\subsubsection{On $j_p=1$ at Eisenstein ideals}

As mentioned in the previous section, by the work of Ribet--Takahashi \cite{ribet1997parametrizations}, it is known that, for optimal elliptic curve quotients, if the residual representation attached to a maximal ideal $\mathfrak m$ is irreducible and $M$ is not prime, then the cokernel of
$
\phi:\Phi_p(J^D(M))
\longrightarrow
\Phi_p(A^D(M))
$
vanishes after localization at $\mathfrak m$. Khare \cite{khare2003isomorphisms} obtained a higher-dimensional analogue under additional hypotheses.
In this section, we investigate the local surjectivity of $\phi$ at exceptional non-Eisenstein maximal ideals and at Eisenstein maximal ideals, and its relationship with multiplicity one. This connection was previously studied by Emerton in a more general setting\cite[Corollary 3.6]{emerton2003optimal}.

Before stating our main theorem, we introduce the notions of the \emph{lattice congruence ideal} and the \emph{lattice intersection ideal} associated with $A^D(M)$ (cf. Definition 3.1).

\begin{Def}
	Consider the decomposition $J^D(M)=A^D(M)+B^D(M)$, where $B^D(M)=IJ^D(M)$ with $I=\mathrm{Ker}(\pi_{A^D(M)}: \mathbb{T}^D(M) \rightarrow \mathrm{End}(A^D(M)))$. Similarly, one has $\pi_{B^D(M)}: \mathbb{T}^D(M) \rightarrow \mathrm{End}(B^D(M))$, and let $\mathbb{T}^D(M)_{A^D(M)}$ and $\mathbb{T}^D(M)_{B^D(M)}$ be $\pi_{A^D(M)}(\mathbb{T}^D(M))$ and $\pi_{B^D(M)}(\mathbb{T}^D(M))$, respectively. Let $p$ be a prime dividing $DM$. One can similarly consider the natural maps
	\[
	\pi_{A^D(M)}^{\mathcal{X}}: \mathbb{T}^D(M) \rightarrow \mathrm{End}(\mathcal{X}_p(J^D(M))[I])
	\]
	and
	\[
	\pi_{B^D(M)}^{\mathcal{X}}: \mathbb{T}^D(M) \rightarrow \mathrm{End}(\mathcal{X}_p(J^D(M))[\mathrm{Ann}_{\mathbb{T}^D(M)}(I)]).
	\]
	
	\begin{enumerate}
		\item The \emph{lattice congruence ideal} associated with $A^D(M)$ is
		\[
		R^D(M)^{\mathcal{X}} := \pi_{A^D(M)}^{\mathcal{X}}(\ker(\pi_{B^D(M)}^{\mathcal{X}}));
		\]
		
		\item The \emph{lattice intersection ideal} associated with $A^D(M)$ is
		\[
		S^D(M)^{\mathcal{X}} := \mathbb{T}^D(M)_{A^D(M)} \cap \mathrm{End}_{\mathbb{T}^D(M)}(\mathcal{X}_p(J^D(M))),
		\]
		where $\mathrm{End}_{\mathbb{T}^D(M)}(\mathcal{X}_p(J^D(M)))$ is the subring of $\mathbb{T}^D(M)$-equivariant endomorphisms, and the intersection is taken in
		\[
		\mathrm{End}(\mathcal{X}_p(J^D(M))[I]) \oplus \mathrm{End}(\mathcal{X}_p(J^D(M))[\mathrm{Ann}_{\mathbb{T}^D(M)}(I)])
		\]
		(note that there is a natural $\mathbb{T}^D(M)$-equivariant injection
		\[
		\mathrm{End}_{\mathbb{T}^D(M)}(\mathcal{X}_p(J^D(M)))\rightarrow
		\mathrm{End}(\mathcal{X}_p(J^D(M))[I]) \oplus
		\mathrm{End}(\mathcal{X}_p(J^D(M))[\mathrm{Ann}_{\mathbb{T}^D(M)}(I)]).
		\]
	\end{enumerate}
\end{Def}

Let $\mathbb{T}^D(M)^{\mathcal{X}}$ denote the saturation of $\mathbb{T}^D(M)$ in $\mathrm{End}_{\mathbb{T}^D(M)}(\mathcal{X}_p(J^D(M)))$, i.e.,
\[
(\mathbb{T}^D(M)\otimes \mathbb{Q}) \cap \mathrm{End}_{\mathbb{T}^D(M)}(\mathcal{X}_p(J^D(M))).
\]

We have the following lemma:

\begin{lem}
	There is a $\mathbb{T}^D(M)$-equivariant injection
	\[
      S^D(M)^{\mathcal X}/R^D(M)^{\mathcal X}
      	\hookrightarrow \mathbb{T}^D(M)^{\mathcal{X}}/\mathbb{T}^D(M).
	\]
\end{lem}

\begin{proof}
	Note that inside
	\[
	\mathrm{End}(\mathcal{X}_p(J^D(M))[I])\oplus \mathrm{End}(\mathcal{X}_p(J^D(M))[\mathrm{Ann}_{\mathbb{T}^D(M)}(I)]),
	\]
	we have $R^D(M)^{\mathcal{X}} = S^D(M)^{\mathcal{X}} \cap \mathbb{T}^D(M)$ and
	\[
	S^D(M)^{\mathcal{X}} = \mathbb{T}^D(M)_{A^D(M)} \cap \mathbb{T}^D(M)^{\mathcal{X}}.
	\]
	Hence,
	\[
	S^D(M)^{\mathcal{X}}/R^D(M)^{\mathcal{X}}
	= S^D(M)^{\mathcal{X}}/(S^D(M)^{\mathcal{X}} \cap \mathbb{T}^D(M))
	\cong (S^D(M)^{\mathcal{X}}+\mathbb{T}^D(M))/\mathbb{T}^D(M)
	\hookrightarrow \mathbb{T}^D(M)^{\mathcal{X}}/\mathbb{T}^D(M).
	\]
\end{proof}

\begin{rem}
	If $p$ divides $D$, since $J^D(M)$ has purely toric reduction at $p$, it follows that $\mathbb{T}^D(M)_{A^D(M)}$ is a subring of $\mathrm{End}(\mathcal{X}_p(J^D(M))[I])$, and similarly for $\mathbb{T}^D(M)_{B^D(M)}$. Also, note that the maps $\pi_{A^D(M)}^{\mathcal{X}}$ and $\pi_{B^D(M)}^{\mathcal{X}}$ factor through $\pi_{A^D(M)}$ and $\pi_{B^D(M)}$ via the natural inclusions. Hence, in this case, the congruence ideal and the lattice congruence ideal are equal, i.e.,
	\[
	R^D(M)^{\mathcal{X}} = R^D(M).
	\]
\end{rem}

Now we restate and prove Theorem $D$ of the introduction:
\begin{thmD}
	Let $A^D(M)$ be an elliptic curve attached to a newform $f$,
	viewed as a subvariety of $J^D(M)$, and let $p$ be a prime
	dividing $D$. Let $l$ be a prime number and let
	$\mathfrak{m}$ be the maximal ideal of $\mathbb{T}^D(M)$
	attached to $A^D(M)[l]$. If
	\[
	\mathcal{X}_p(J^D(M))_{\mathfrak{m}}
	\cong \mathbb{T}^D(M)_{\mathfrak{m}},
	\]
	then the map
	\[
	\phi: \Phi_p(J^D(M)) \rightarrow \Phi_p(A^D(M))
	\]
	is surjective after localization at $\mathfrak{m}$.
\end{thmD}

\begin{proof}
    Let $n$ and $r$ denote the denominators of $e$ in $\mathrm{End}(J^D(M))$ and $\mathrm{End}(\mathcal{X}_p(J^D(M)))$, respectively. By definition, $n=\delta^D(M)$, and by replacing the rigid analytic lattice $\Lambda$ of \cite{papikian2016optimal} by $\mathcal{X}_p(J^D(M))$ and using Grothendieck's monodromy pairing~\cite[14.2.5]{grothendieck1972modeles} in place of the rigid analytic version used in $loc.~cit.$, one obtains formulas (3.7) and (3.9) of \cite[p.~1370]{papikian2016optimal}, which give
	\[
	j_p=\frac{n}{r}.
	\]
	
	Now let $\mathfrak{m}$ be the maximal ideal of $\mathbb{T}^D(M)$ of residue characteristic $l$ containing $I_f$. By assumption,
	\[
	\mathcal{X}_p(J^D(M))_{\mathfrak{m}}
	\cong
	\mathbb{T}^D(M)_{\mathfrak{m}}.
	\]
	Since $\mathrm{End}(\mathcal{X}_p(J^D(M)))$ acts faithfully on $\mathcal{X}_p(J^D(M))$ and
	\[
	\mathbb{T}^D(M)
	\subset
	\mathrm{End}(\mathcal{X}_p(J^D(M))),
	\]
	it follows that
	\[
	\mathbb{T}^D(M)^{\mathcal{X}}_{\mathfrak{m}}
	\cong
	\mathbb{T}^D(M)_{\mathfrak{m}}.
	\]
	By Lemma~9.4,
	\[
	R^{D}(M)^\mathcal{X}_{\mathfrak{m}}
	=
	S^{D}(M)^\mathcal{X}_{\mathfrak{m}}
	\]
	while Proposition~3.4 and the assumed freeness imply
	\[
	R^D(M)_{\mathfrak{m}}
	=
	S^D(M)_{\mathfrak{m}}.
	\]
	Remark~9.5 therefore yields
	\[
	S^{D}(M)^\mathcal{X}_{\mathfrak{m}}
	=
	R^{D}(M)^\mathcal{X}_{\mathfrak{m}}
	=
	R^D(M)_{\mathfrak{m}}
	=
	S^D(M)_{\mathfrak{m}}.
	\]
	By definition, the image of $e$ has order $r$ in
	$
	\mathbb{T}^D(M)_{A^D(M)}/S^{D}(M)^\mathcal{X},
	$
	and order $\delta^D(M)$ in
	$
	\mathbb{T}^D(M)_{A^D(M)}/S^D(M).
	$
	Since the two ideals agree after localization at $\mathfrak{m}$, we conclude that
	\[
	\text{Coker}(\phi)_\mathfrak{m}=0.
	\]
\end{proof}

%% file: Appendix.tex
\section{Appendix: Source Code for the Computations}

In this appendix, we provide the Sage and Magma code used to perform the computations reported in Remark~3.2 and Sections~9.1 and~9.2.1. These implementations are adapted from the code appearing in the appendix of \cite{deines2014shimura}. All computations were carried out using Magma V2.23-10 and SageMathCell.

\subsection{Magma Code for Monodromy Computations}

The following Magma code computes the quantities $h_p/i_p$ appearing in Section 9.2.1.

\begin{lstlisting}

// MonodromyFromAlphaBeta function
// =======================================================
function MonodromyFromAlphaBeta(D, M, p, E)

ZZ := Integers();
Zp := ideal<ZZ | p>;
Q := QuaternionAlgebra(D);
QM := MaximalOrder(Q);
OM := Order(QM, M);

// p-conjugate
function pConjugate(O, Ip)
Opp, phi, mpp := pMatrixRing(O, Ip);
alpha := Opp![[0,1],[mpp(Generator(Ip)),0]];
I := rideal<O | [alpha@@phi, Generator(Ip)]>;
return LeftOrder(I);
end function;

OMg := pConjugate(OM, Zp);
OMp := Order(OM, p);
OMps := OM meet OMg;
X := LeftIdealClasses(OM);
Y := LeftIdealClasses(OMg);
Z := LeftIdealClasses(OMps);

B := BrandtModule(OMps);
InnerProductMatrix(B);

BRI := [Conjugate(BI) : BI in Z];

BM := BrandtModule(OM);
BRIM := [Conjugate(BI) : BI in X];

BMg := BrandtModule(OMg);
BRIMg := [Conjugate(BI) : BI in Y];

function RightIdealClass(I, BRList)
for idx := 1 to #BRList do
if IsIsomorphic(BRList[idx], I) then
return BRList[idx], idx;
end if;
end for;
error "Right ideal not found!";
end function;

function Alpha(I)
I := rideal<OM | [b : b in Basis(I)]>;
return RightIdealClass(I, BRIM);
end function;

function Alphag(I)
I := rideal<OMg | [b : b in Basis(I)]>;
return RightIdealClass(I, BRIMg);
end function;

Mat := [];
for RI in BRI do
AI, ai := Alpha(RI);
BI, bi := Alphag(RI);
v := [0 : i in [1..#BRI]]; v[ai] := 1;
w := [0 : i in [1..#BRI]]; w[bi] := 1;
Append(~Mat, v cat w);
end for;

ABMat := Matrix(Mat);
ABBas := Basis(Kernel(ABMat));
BBas := [B!Eltseq(ab) : ab in ABBas];

coefs := [i : i in PrimesUpTo(100) | Integers()!(p*D*M) mod i ne 0];

n := Dimension(B);
V := Lattice(IdentityMatrix(Integers(), n));

KerIntersection := V;

for c in coefs do
T := HeckeOperator(B, c);
f := Newform(E);
coef := Coefficient(f, c);
TT := Evaluate(MinimalPolynomial(coef), T);
d := LCM([Denominator(x) : x in Eltseq(TT)]);
AZ := Matrix(Integers(), d*(TT));
Kmod := Kernel(AZ);                 // ModTupRng
Klat := sub< V | Basis(Kmod) >;     // Lat
KerIntersection := KerIntersection meet Klat;
end for;

if Dimension(KerIntersection) eq 0 then
error "Intersection of Hecke kernels is trivial!";
end if;

Ba := Basis(KerIntersection);
n := #Ba;
for i in [1..n] do
g := Ba[i];
S := [ Denominator(x) : x in Eltseq(g) ];
hh := LCM(S);
Ba[i] := Ba[i]*hh;
end for;

BaInt := [ B![Integers()!x : x in Eltseq(g)] : g in Ba ];

hMat := Matrix(n, n, [ InnerProduct(BaInt[i], BaInt[j]) : i in [1..n], j in [1..n] ]);
m := Dimension(Kernel(ABMat));
List := [];
for I in Subsets({1..m}, n) do
HH := Matrix(n, n, [ InnerProduct(BaInt[i], BBas[j]) : i in [1..n] , j in I ]);
if Determinant(HH) ne 0 then
Append(~List, Integers()!(Determinant(HH)));
end if;
end for;

hh := GCD(List);
return Determinant(hMat)/hh;
end function;

// Main loop
for N in [1..500] do

if not IsSquarefree(N) then
continue;
end if;

if #PrimeDivisors(N) lt 2 then
continue;
end if;

printf "N = %o\n", N;

CM := CuspForms(Gamma0(N));
fs := Newforms(CM);

for fam in fs do

try

E := ModularAbelianVariety(fam[1]);
dimE := Dimension(E);

printf "  dim(E) = %o\n", dimE;

for D in Divisors(N) do

if D eq 1 then
continue;
end if;

if (#PrimeDivisors(D) mod 2) ne 0 then
continue;
end if;

M := N div D;

printf "  D = %o, M = %o\n", D, M;

for p in PrimeDivisors(D) do

DD := D div p;

try

mon := MonodromyFromAlphaBeta(
DD,
M,
p,
E
);

printf
"N=%o D=%o p=%o DD=%o M=%o dim=%o mon=%o\n",
N, D, p, DD, M, dimE, mon;

catch err

printf
"ERROR: N=%o D=%o p=%o DD=%o M=%o\n",
N, D, p, DD, M;

print err;

end try;

end for;

end for;

catch err

printf "FAILED NEWFORM AT N=%o\n", N;
print err;

end try;

end for;

end for;
\end{lstlisting}

\subsection{Sage Code: Generalized $D$-Degeneracy Computations}

The following Sage code computes the modular invariants $\tilde{\delta}^1_D(N)$ and $\tilde{r}^1_D(N)$ associated to modular abelian varieties $A^1_D(N)$ (Section 9.1).

\begin{lstlisting}
def D_degcong(NN):

Dlist = NN.divisors()

Q = J0(NN)

J = Q.modular_symbols(-1)

Jnew = J.new_submodule()

JD = J.decomposition()

for i in range(len(Jnew.decomposition())):

A = Jnew[i]

Wnew = Q.new_subvariety()

#WDnew = Sym.abelian_variety()

#S = Sym.decomposition()

E = Wnew[i]

print(E)

for DD in Dlist:

Sym = J

for p in DD.prime_divisors():

Sym = Sym.intersection(J.new_submodule(p))

B = A.complement()

AB = B.intersection(Sym)

prec = max(A.hecke_bound(), B.hecke_bound())

V = A.q_expansion_module(prec, ZZ)

W = AB.q_expansion_module(prec, ZZ)

Z = Sym.q_expansion_module(prec, ZZ)

K = V + W

C = V.basis()

F = W.basis()

rows = []

# First basis elements from A

for x in C:

v = Z.coordinate_vector(x)

rows.append(list(v))

# Then basis elements from B

for x in F:

v = Z.coordinate_vector(x)

rows.append(list(v))

# Construct the matrix
V = matrix(ZZ, rows)

#abs_det = abs(V.determinant())

#print(V,abs_det)

D, U, S = V.smith_form()

# The diagonal of D contains the invariant factors

invariant_factors = D.diagonal()
  
exponent = max(invariant_factors)

print(DD, "-Congruence exponent =", exponent, "and", DD, "-Congruence number =", D.determinant())

Jn = Q

for p in DD.factor():

# Take sum of all new subvarieties at p except E

Jpnew = sum([jnew for jnew in Q.new_subvariety(p[0]).decomposition() if jnew != E])

if Jpnew == 0:

Jn = Q.zero_subvariety()

break

# Intersect with previous Jnew

Jn = Jpnew.intersection(Jn)[1]

print(DD, "-modular exponent =", Jn.intersection(E)[0].exponent(), DD, "-modular number =", Jn.intersection(E)[0].order())
\end{lstlisting}

\textbf{Acknowledgments.}
I am very grateful to Patrick Allen for many fruitful discussions and his continued encouragement. I also thank John Voight for his guidance with the Magma code, Amod Agashe for helpful conversations, and Mihran Papikian for useful comments.

For the preparation of this manuscript, I used ChatGPT (GPT-5.6 and GPT-6) solely for grammar correction and stylistic editing. The mathematical ideas, arguments, and results are entirely my own, and I take full responsibility for the mathematical content of this paper.